\documentclass[11pt]{amsart}
\usepackage{amsmath,amsfonts,amsthm,amssymb,setspace}
\usepackage[abs]{overpic}
\usepackage[margin=1in]{geometry}
\numberwithin{equation}{section}
\newtheorem{thm}[equation]{Theorem}
\newtheorem{lem}[equation]{Lemma}
\newtheorem{prop}[equation]{Proposition}
\newtheorem{cor}[equation]{Corollary}
\newtheorem*{cor*}{Corollary}
\newtheorem*{claim*}{Claim}

\theoremstyle{definition}
\newtheorem{defn}[equation]{Definition}

\newtheorem*{exs}{Examples}
\theoremstyle{remark}
\newtheorem*{rem}{Remark}
\newcommand{\zz}{\mathbb{Z}}
\newcommand{\qq}{\mathbb{Q}}
\newcommand{\rr}{\mathbb{R}}

\newcommand{\PP}{\mathcal{P}}
\newcommand{\sss}{\mathcal{S}}

\newcommand{\GG}{\mathcal{G}}

\newcommand{\Cab}{C_{(a,b)}}
\newcommand{\Ccd}{C_{(c,d)}}
\newcommand{\gen}{\operatorname{gen}}
\newcommand{\id}{\operatorname{id}}
\newcommand{\curv}{\operatorname{curv}}
\newcommand{\Symm}{\operatorname{Symm}}
\newcommand{\Stab}{\operatorname{Stab}}
\newcommand{\SL}{\operatorname{SL}}
\newcommand{\GL}{\operatorname{GL}}
\newcommand{\PSL}{\operatorname{PSL}}
\newcommand{\PGL}{\operatorname{PGL}}
\renewcommand{\labelenumi}{(\roman{enumi})}
\usepackage{graphicx}
\usepackage{verbatim}
\usepackage{subfigure}

\pdfoutput=1

\author{Michael Ching}
\author{John R. Doyle}

\thanks{Research for this paper was supported by NSF VIGRE grant DMS-0738586.}

\title{Apollonian circle packings of the half-plane}

\begin{document}

\begin{abstract}
We consider Apollonian circle packings of a half Euclidean plane. We give necessary and sufficient conditions for two such packings to be related by a Euclidean similarity (that is, by translations, reflections, rotations and dilations) and describe explicitly the group of self-similarities of a given packing. We observe that packings with a non-trivial self-similarity correspond to positive real numbers that are the roots of quadratic polynomials with rational coefficients. This is reflected in a close connection between Apollonian circle packings and continued fractions which allows us to completely classify such packings up to similarity.\\


\end{abstract}

\maketitle

\section{Introduction}\label{sec:intro}

A circle packing in $\rr^2$ is a set of circles in the plane whose interiors (suitably interpreted) are mutually disjoint. An \textbf{Apollonian circle packing} $\PP$ has the property that for any three mutually tangent circles in $\PP$, the two circles in the plane that are tangent to all three of them also lie in $\PP$. Note that our notion of circle includes straight lines where we consider parallel lines to be tangent at infinity.

These types of circle packings have been extensively studied by Graham, et al \cite{graham:2003,graham:2005,graham:2006a,graham:2006b}, with a focus on those packings for which all the circles have integer curvatures.

There are four basic shapes that an Apollonian packing may take, and these are illustrated in Figure~\ref{fourtypes}. A \textbf{bounded} Apollonian packing (Figure~\ref{bddpack}) is a packing $\PP$ for which a single circle in $\PP$ bounds the entire packing. Here the `interior' of the bounding circle is the \emph{unbounded} component of its complement.

A \textbf{half-plane} packing (Figure~\ref{hppack}) is an Apollonian packing $\PP$ for which at least one of its circles is a straight line. The line partitions the plane into two half-planes: one is `packed' by $\PP$, while the other is the `interior' of the line. A special type of half-plane packing is a \textbf{strip} packing (Figure~\ref{strippack}), in which two of the circles are (necessarily parallel) lines and the remaining circles lie in the strip between them.

An \textbf{unbounded} packing (Figure~\ref{unbddpack}) is an Apollonian packing which contains no bounding circle and no straight line.

\begin{figure}[h]
	\centering
    \subfigure[Bounded]{\label{bddpack}\includegraphics[scale=.4]{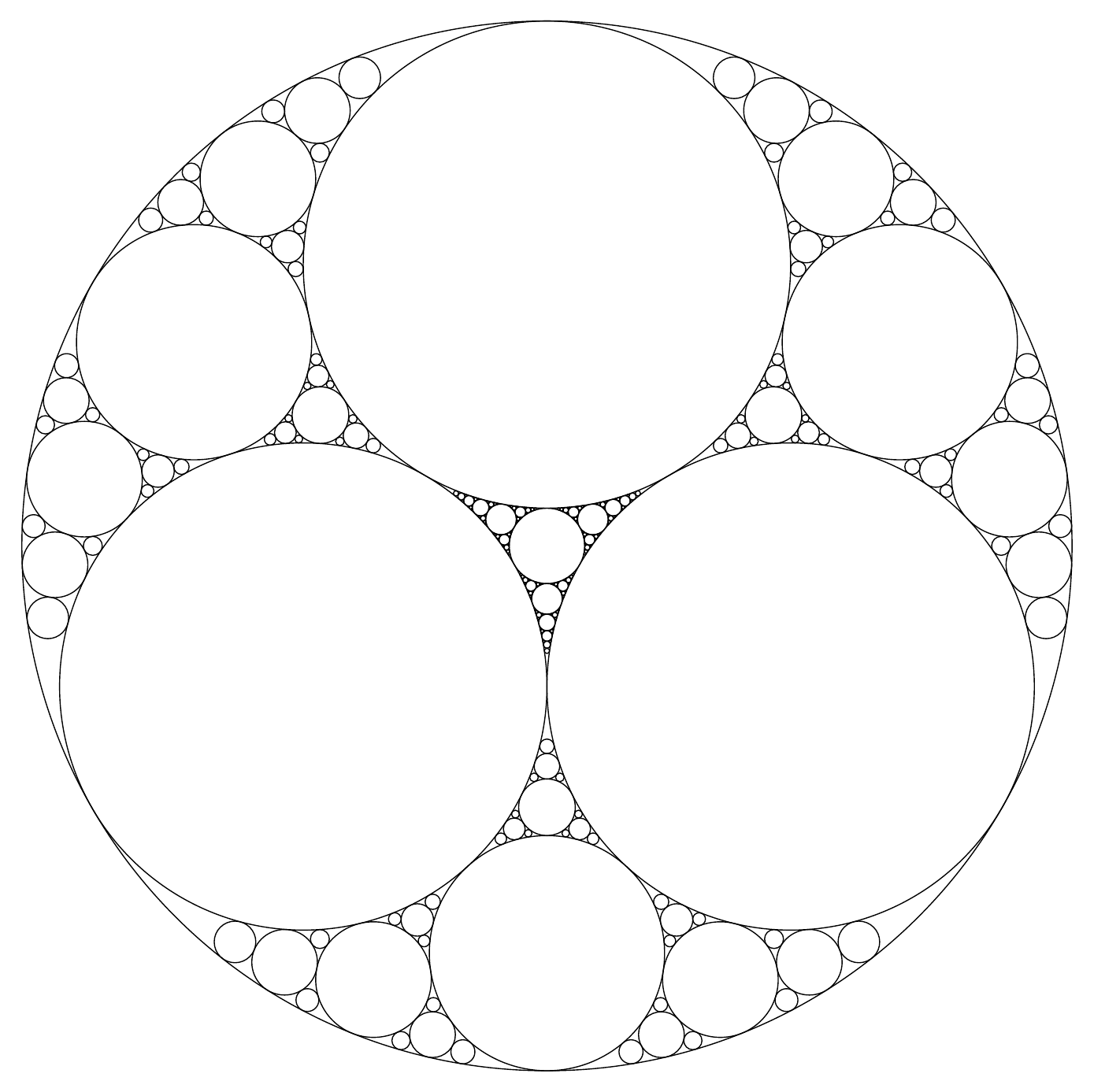}} \hspace{20mm}
    \subfigure[Half-plane]{\label{hppack}\includegraphics[scale=.4]{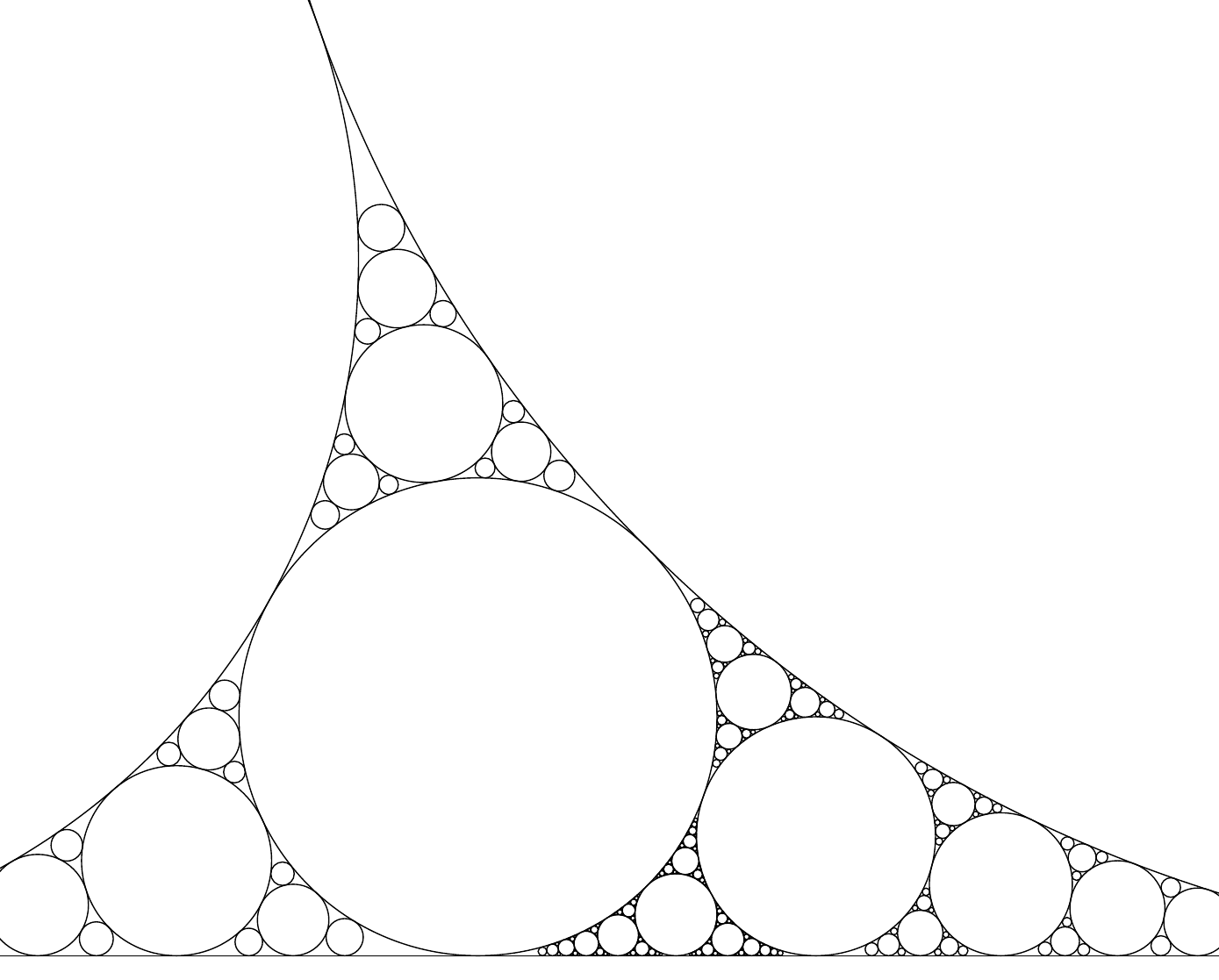}}\\
    \subfigure[Strip]{\label{strippack}\includegraphics[scale=.4]{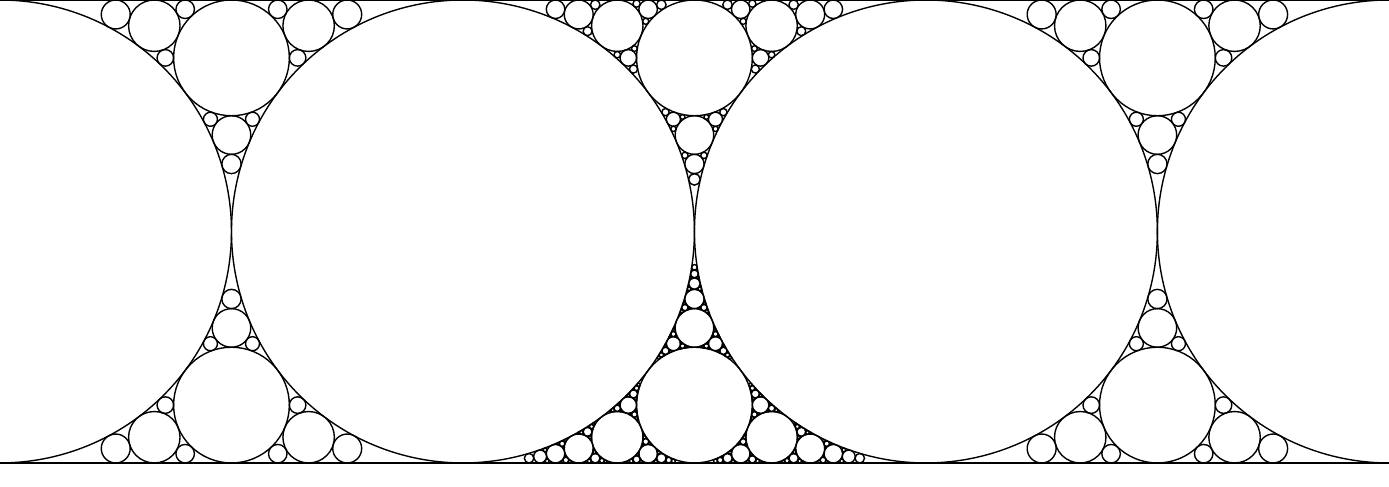}}
\hspace{20mm}
    \subfigure[Unbounded]{\label{unbddpack}\includegraphics[scale=.4]{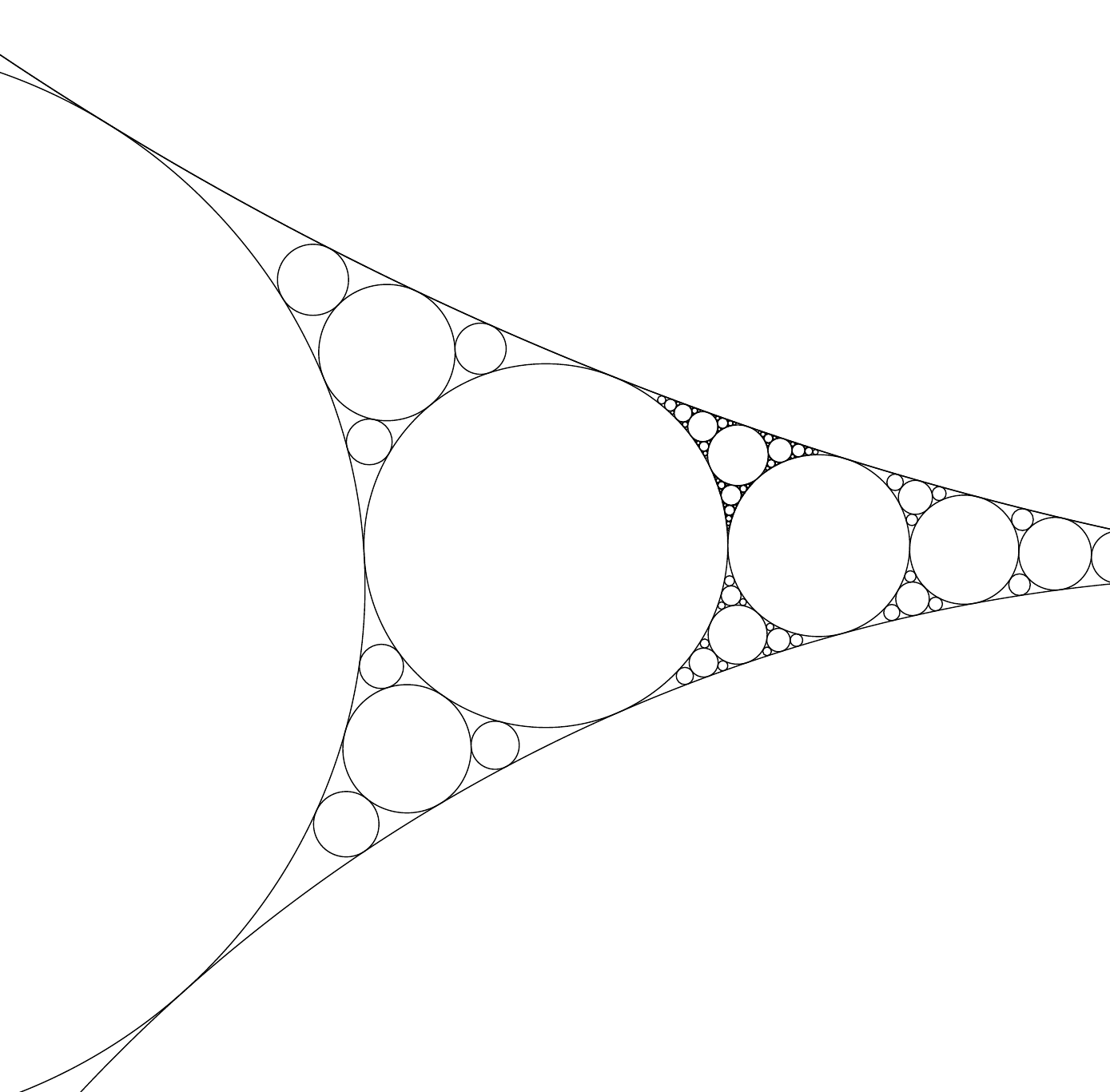}}
	\caption{Different boundedness properties for Apollonian packings.}
	\label{fourtypes}
\end{figure}

Stereographic projection allows us to relate circle packings in the plane to those on the sphere. The four possible configurations in Figure \ref{fourtypes} correspond to projection from (a) an interior point, (b) a point on only one circle, (c) a tangency point, and (d) a point not on or inside any circle, respectively.

In this paper, we consider the similarity relation on Apollonian circle packings. A \textbf{similarity} is a transformation of the Euclidean plane that preserves ratios of lengths. Such a transformation is necessarily a composite of a translation, rotation, reflection and/or dilation. Two packings $\PP$ and $\PP'$ are \textbf{similar} if there is a similarity of the plane that takes circles in $\PP$ bijectively to circles in $\PP'$. There may be non-trivial similarities from $\PP$ to itself, in which case we say that $\PP$ is \textbf{self-similar} and we consider its group of self-similarities.

Our results concern only half-plane packings. We give a necessary and sufficient condition for two such packings to be similar, a classification of the self-similar packings, and a description of all of the self-similarity groups.

We also answer the more specific question of whether two packings are similar via an \textbf{orientation-preserving} similarity (that is, one with positive determinant) or via an \textbf{orientation-reversing} similarity. Our classification tells us which packings possess an orientation-reversing self-similarity.

We observe that any half-plane packing $\PP$ is similar to a packing containing three circles in the configuration shown in Figure~\ref{P_alpha}, where $L$ is the $x$-axis, and $\alpha^2$ and $1$ refer to the curvatures of the circles they label. For $\alpha > 0$, we define $\PP_{\alpha}$ to be the unique Apollonian circle packing containing that configuration. Because each half-plane packing $\PP$ is similar to such a packing, we restrict our attention to studying the packings $\PP_{\alpha}$, and we state results in terms of this particular class of packings.

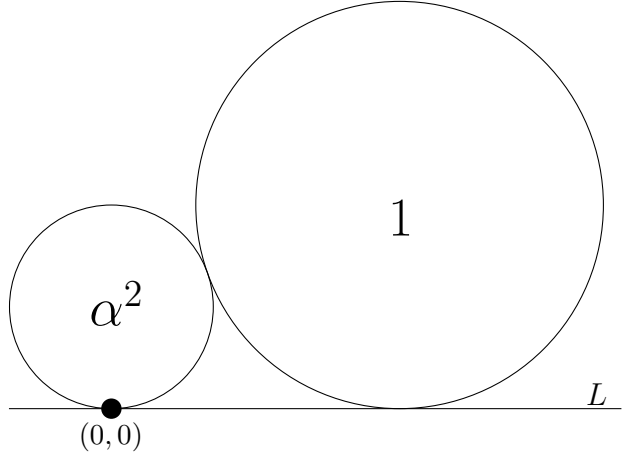
\begin{figure}[h]
	\begin{overpic}[scale = 1]{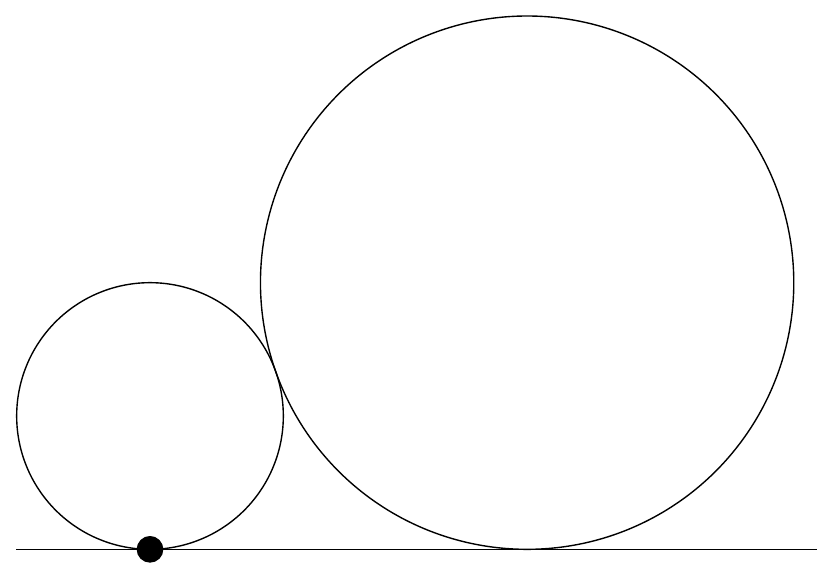}
	\put(35,40){\huge $\alpha^2$}
	\put(148,74){\huge $1$}
	\put(31,-5){$(0,0)$}
	\put(223,10){\large $L$}
	\end{overpic}
	\caption{The generating triple for the packing $\PP_{\alpha}$. The numbers $\alpha^2$ and $1$ represent the curvatures of their respective circles, and $L$ is a straight line.}
  \label{P_alpha}
\end{figure}

Our first result relates similarities of half-plane packings to elements of the projective general linear group $\PGL_2(\zz)$. This is the quotient of the group $\GL_2(\zz)$ of invertible $2 \times 2$-matrices with integer entries, by the subgroup consisting of $\pm I$ where $I$ denotes the identity matrix. We also refer to the subgroup $\PSL_2(\zz)$ consisting of those elements whose underlying matrices have determinant $1$.

\begin{thm} \label{thm1-intro}
Let $\alpha,\beta > 0$. There is a bijection between the set of similarities preserving the $x$-axis that map $\PP_{\beta}$ to $\PP_{\alpha}$ and the set of elements
$\begin{bmatrix}
a & b\\
c & d
\end{bmatrix}
\in \PGL_2(\zz)$
 such that $\dfrac{a\alpha + b}{c\alpha + d} = \beta$. The similarity is orientation-preserving if and only if the corresponding element is in $\PSL_2(\zz)$. In particular, $\PP_{\alpha}$ and $\PP_{\beta}$ are similar (resp. similar via an orientation-preserving similarity) if and only if there exist integers $a,b,c,d$, with $ad - bc = \pm 1$ (resp. $+1$), such that $\dfrac{a\alpha + b}{c\alpha + d} = \beta$.
\end{thm}

Taking $\alpha = \beta$ in Theorem \ref{thm1-intro} helps us to calculate the self-similarity groups. We write $\Symm(\PP)$ for the self-similarity group of the packing $\PP$, and we write $\Symm_+(\PP)$ for the subgroup of $\Symm(\PP)$ consisting of orientation-preserving self-similarities of $\PP$.

\begin{thm}\label{symmgroup}
Let $\alpha > 0$. Then:
\begin{enumerate}
  \item If $\alpha \in \mathbb{Q}$, then $\PP_{\alpha}$ is a strip packing and
  \[ \Symm(\PP_{\alpha}) \cong D_{\infty} \times \mathbb{Z}/2\mathbb{Z} \]
  with subgroup
  \[ \Symm_+(\PP_{\alpha})\cong D_{\infty}, \]
  where $D_{\infty}$ denotes the infinite dihedral group.
  \item If $\alpha$ is quadratic over $\mathbb{Q}$, then
  \[ \Symm(\PP_{\alpha}) \cong \mathbb{Z}. \]
  Let $D$ denote the discriminant of the primitive integral polynomial with root $\alpha$. Then the subgroup $\Symm_+(\PP_{\alpha})$ is:
  \begin{itemize}
    \item equal to $\Symm(\PP_{\alpha})$ if the Pell equation $x^2 - Dy^2 = -4$ has no integral solution for $(x,y)$;
    \item the index $2$ subgroup of $\Symm(\PP_{\alpha})$ if $x^2 - Dy^2 = -4$ does have an integral solution.
  \end{itemize}
  \item Otherwise
  \[ \Symm_+(\PP_{\alpha}) = \Symm(\PP_{\alpha}) = 1. \]
\end{enumerate}
\end{thm}

As one might imagine from the form of Theorem \ref{symmgroup}, there is a striking connection between half-plane Apollonian circle packings and continued fractions which we describe in \S\ref{sec:contfrac}. In particular, we have the following result.

\begin{thm}\label{cfeequiv}
Let $\alpha,\beta > 0$. The packings $\PP_{\alpha}$ and $\PP_{\beta}$ are similar if and only if the continued fraction expansions of $\alpha$ and $\beta$ are eventually equal (that is, become equal when initial segments, of possibly different lengths, are removed from each).
\end{thm}

As a consequence of Theorem \ref{cfeequiv}, we obtain the following classification of the self-similar half-plane packings.

\begin{thm} \label{cfclass}
The similarity classes of self-similar half-plane (non-strip) packings are in a one-to-one correspondence with the finite, non-repeating sequences of positive integers, up to cyclic permutations. Here ``non-repeating" means that the sequence cannot be realized as a concatenation of multiple copies of a shorter sequence. The packing has an orientation-reversing self-similarity if and only if the length of the corresponding sequence is odd.
\end{thm}

In Figures \ref{fig:selfsim1}-\ref{fig:selfsim12}, at the end of the paper, we show the self-similar half-plane packings corresponding to the sequences $(1)$, $(2)$, $(3)$ and $(1,2)$.

Here is a quick outline. In \S\ref{sec:acp} we give a precise definition of Apollonian circle packings and establish some of their basic properties. The main part of that section is then to describe a labelling system for the circles in a half-plane packing that are tangent to the $x$-axis. In \S\ref{sec:selfsim} we relate those labels to the curvatures of the circles and use this relationship to prove Theorems \ref{thm1-intro} and \ref{symmgroup}. In \S\ref{sec:contfrac} we examine the connection between half-plane packings and continued fractions which we use to prove Theorems \ref{cfeequiv} and \ref{cfclass}.

\subsection*{Acknowledgements}

This project grew out of a VIGRE Research Group on Circle Packings at the University of Georgia. We would like to thank Sa'ar Hersonsky for his leadership and suggestions. Michael Berglund and Chris Pryby were involved in much of the initial work on the project and this paper would not exist without their contributions. Other members of the group, including Jennifer Ellis, provided useful feedback.

\section{Apollonian Circle Packings}\label{sec:acp}

We begin with a precise definition of an Apollonian circle packing.

\begin{defn} \label{apollonian}
For the purposes of this paper, a \textbf{circle} in $\rr^2$ is either a circle or a straight line, together with a choice of one of the components of its complement which we refer to as the \textbf{interior} of the circle. Note that what we refer to as the interior of a circle may be the unbounded component of its complement, and the interior of a straight line is one of the two half-planes it determines.

An \textbf{Apollonian circle packing} is a collection $\PP$ of circles in $\rr^2$ with disjoint interiors such that
\begin{enumerate}
  \item there exists a set of three mutually tangent circles in $\PP$;
  \item if a circle $C$ is tangent to three mutually tangent circles that are in $\PP$, then $C$ is also in $\PP$.
\end{enumerate}
\end{defn}

An Apollonian circle packing can be constructed recursively in the following way.

\begin{defn}\label{recursive}
Let $\PP^{(0)}$ be a set of three mutually tangent circles in $\rr^2$ with disjoint interiors. Given $\PP^{(n)}$, we define $\PP^{(n+1)}$ to be the set of circles in $\rr^2$ consisting of $\PP^{(n)}$ together with any circle that is tangent to three mutually tangent circles in $\PP^{(n)}$. The \textbf{Apollonian circle packing generated by $\PP^{(0)}$} is
\begin{equation*}
\PP := \bigcup_{n = 0}^{\infty} \PP^{(n)}
\end{equation*}
It can be seen by an induction argument that the circles in $\PP^{(n)}$ have disjoint interiors, and it follows that $\PP$ is an Apollonian circle packing in the sense of Definition \ref{apollonian}.
\end{defn}

Note that a theorem of Apollonius says that for three mutually tangent circles in $\rr^2$ with disjoint interiors, there are precisely two other circles tangent to all three. Each of these two circles lies in an \emph{interstice} formed by the original three circles. See Figure \ref{fig:interstices}.

\begin{defn}
Let $A$, $B$, and $C$ be three mutually tangent circles in $\rr^2$ with disjoint interiors. The complement of $A \cup B \cup C$ in $\rr^2$ consists of five components --- three of the components are the interiors of the respective circles, and the other two are called the \textbf{interstices} formed by $A$, $B$, and $C$.
\end{defn}

\begin{figure}[h]
\includegraphics[scale=.8,angle=90]{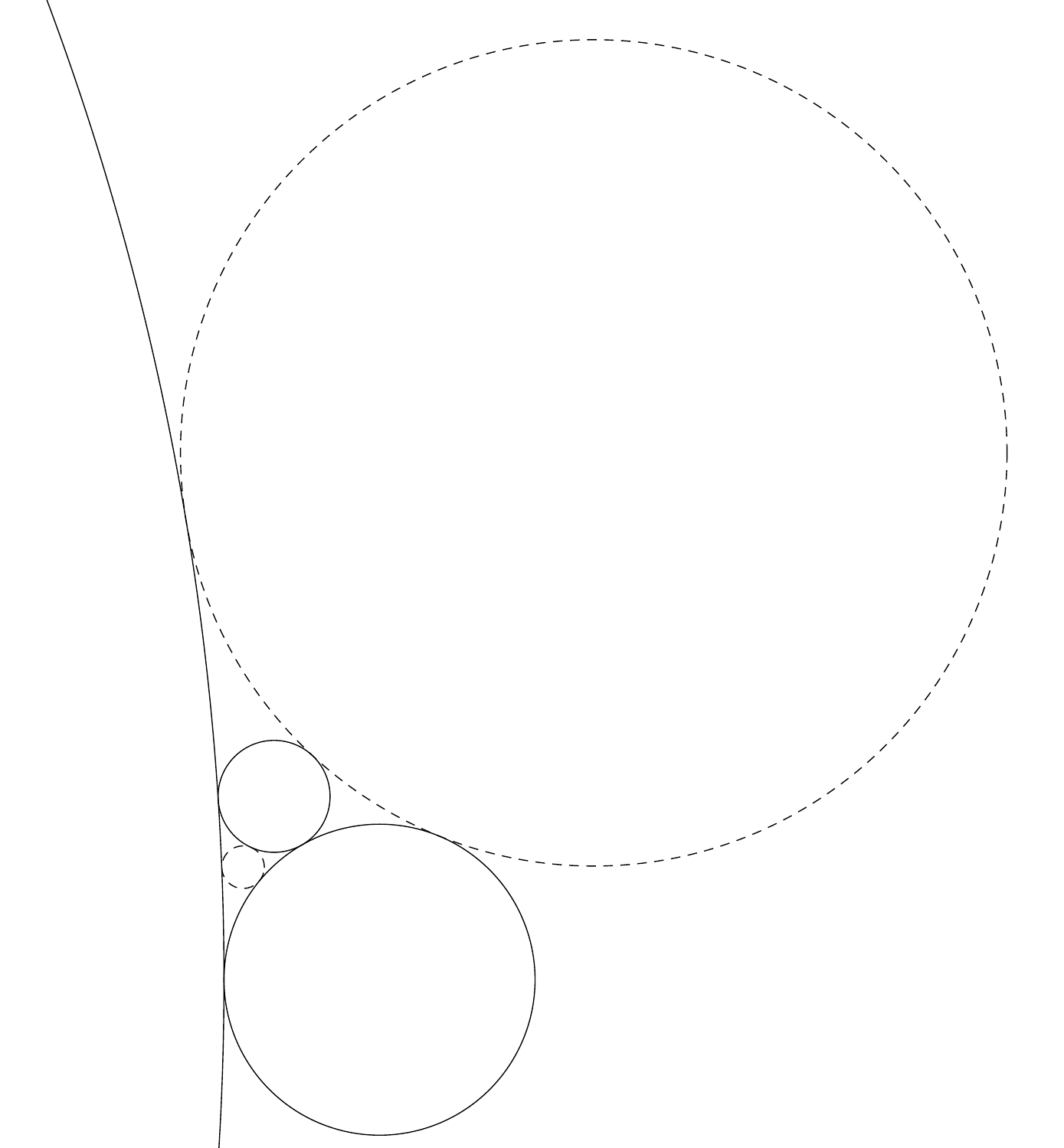}
\caption{The two dashed circles lie in the interstices bounded by the three solid circles.}
\label{fig:interstices}
\end{figure}

\begin{lem}\label{anytriple}
Let $\PP$ be an Apollonian circle packing. Then $\PP$ is generated, in the sense of Definition \ref{recursive}, by any set of three mutually tangent circles in $\PP$.
\end{lem}

\begin{cor}\label{samepack}
If $\PP$ and $\PP'$ are two Apollonian packings with a common triple of mutually tangent circles, then $\PP = \PP'$. \qed
\end{cor}

\begin{proof}[Proof of Lemma \ref{anytriple}]
Since $\PP$ certainly contains the packing generated by any set $\PP^{(0)}$ of three mutually tangent circles, it is sufficient to show that there is no room for any other circles. In particular, this will be true if the complement of the set of interiors of circles in $\PP$ (called the \emph{residual set of $\PP$}) has Lebesgue measure zero. A proof of this fact may be found in \cite[Theorem 4.2]{graham:2005}.
\end{proof}

For us, the point of the recursive construction of Apollonian circle packings is that some of our arguments proceed by induction on the stage at which the circles are created in this process. We therefore make the following definition.

\begin{defn}
Fix a generating triple $\PP^{(0)}$ for the packing $\PP$. The \textbf{generation} of a circle $C \in \PP$ (with respect to $\PP^{(0)}$), denoted by $\gen(C)$, is the unique $n \in \zz_{\ge 0}$ such that $C \in \PP^{(n)} \setminus \PP^{(n-1)}$.
\end{defn}

We now narrow our focus to half-plane packings. Let $\PP$ be a half-plane packing, that is, a packing that contains at least one line $L$. We assume that $L$ coincides with the $x$-axis and that the remaining circles in $\PP$ are in the upper-half plane. (Any half-plane packing is similar to one that satisfies this condition.)

\begin{defn}
Most of our analysis of half-plane packings can be done be focusing on the circles in $\PP$ that are tangent to the line $L$. We define
\[ \PP_L := \{C \in \PP \ | \ C \mbox{ is tangent to } L \}. \]
The `mutually disjoint interiors' requirement of circle packings ensures that no two circles in $\PP_L$ may be tangent to $L$ at the same point. This property allows us to define a total ordering on the set $\PP_L$. We say that \emph{$C$ is to the left of $C'$}, or $C \prec C'$, if the $x$-coordinate of the point of tangency between $C$ and $L$ is \emph{less than} the $x$-coordinate of the point of tangency between $C'$ and $L$. In the case that $\PP$ is a strip packing, with $L'$ the line in $\PP$ which is parallel to $L$, we consider the $x$-coordinate of the point of tangency between $L'$ and $L$ to be $-\infty$; in other words, $L' \prec C$ for all $C \in \PP_L$ with $C \ne L'$.
\end{defn}

\begin{defn}\label{defn:pairinterstice}
Let $X$ and $Y$ be two tangent circles in $\PP_L$, neither of which is a line. Then $\{X,Y,L\}$ is a triple of mutually tangent circles in $\rr^2$ and therefore determines two interstices in the plane. One interstice is bounded, and the other is unbounded; we refer to these as the \textbf{bounded interstice for $X$ and $Y$} and the \textbf{unbounded interstice for $X$ and $Y$} respectively. We say that the circle $C$ \textbf{fills the bounded} (resp. \textbf{unbounded}) \textbf{interstice for $X$ and $Y$} if $C$ is the unique circle in the bounded (resp. unbounded) interstice for $X$ and $Y$ which is tangent to $X$, $Y$, and $L$ (see Figure \ref{fig:interstices2}). Note that, by Definition~\ref{apollonian}, $C$ necessarily lies in $\PP$ and hence also $\PP_L$.
\end{defn}

\begin{rem}
No circle in the bounded interstice for $X$ and $Y$ can be tangent to a circle in the unbounded interstice for $X$ and $Y$.
\end{rem}

\begin{figure}[h]
	\begin{overpic}[scale = .7]{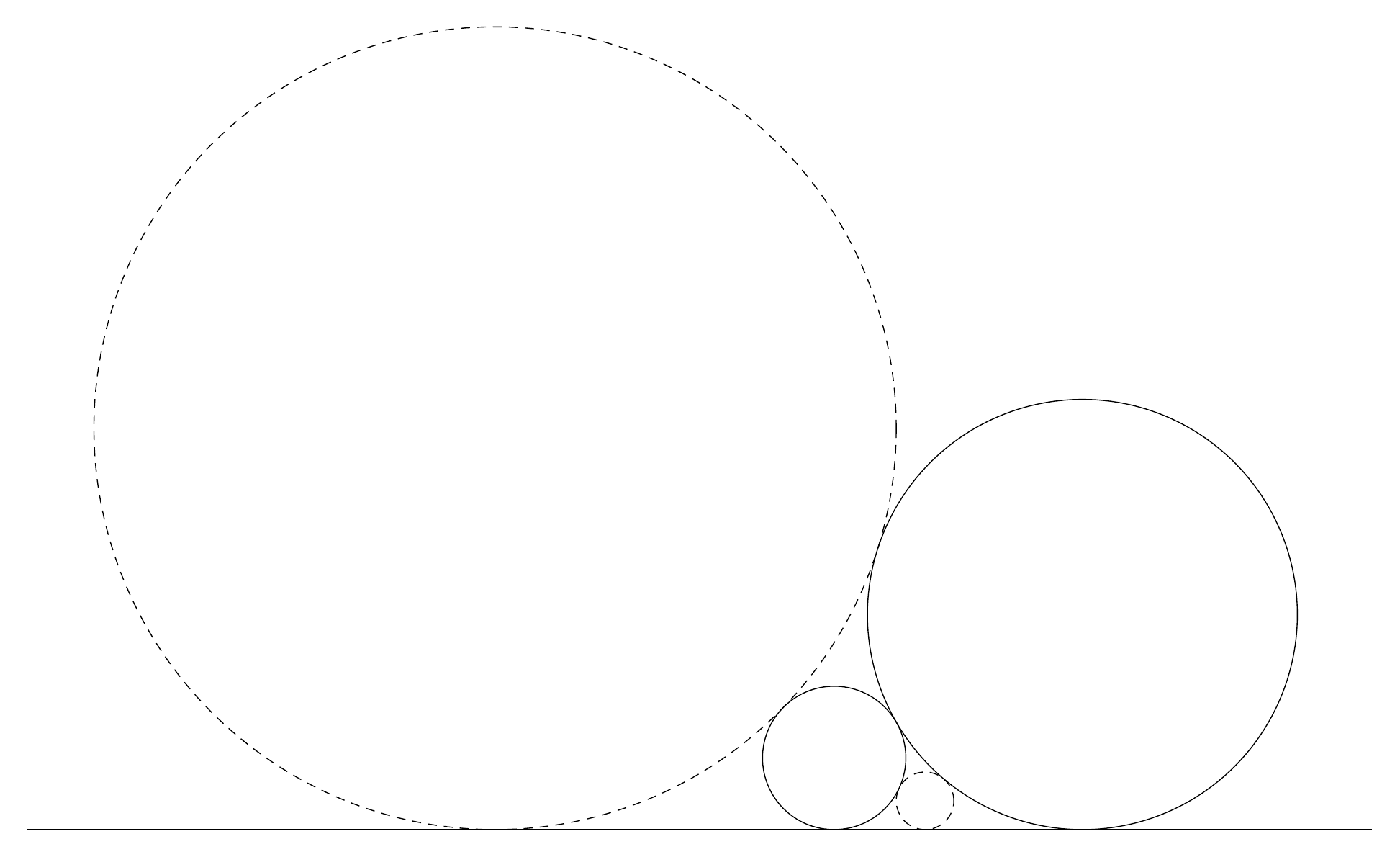}
	\put(225,21){\huge $X$}
	\put(295,54){\Huge $Y$}
	\put(370,10){\huge $L$}
	\end{overpic}
	\caption{The larger dashed circle fills the unbounded interstice for $X$ and $Y$, and the smaller dashed circle fills the bounded interstice.}
  \label{fig:interstices2}
\end{figure}

By Lemma \ref{anytriple} we can view the packing $\PP$ as generated by the triple $\{X,Y,L\}$ in the sense of Definition \ref{recursive} for any pair of tangent circles $X,Y \in \PP_{L}$. For the remainder of this section, we fix a choice of $X$ and $Y$ and assume that $X \prec Y$, that is, $X$ is to the left of $Y$. We also assume that $X$ and $Y$ are actual circles, i.e., neither is a line.

\begin{defn} \label{def:PL+}
It is convenient to divide up the circles in $\PP_L$ according to which interstice they are contained in. We define
\[ \PP_L^+ = \{C \in \PP_L \ | \ X \preceq C \preceq Y\} \]
and
\[ \PP_L^- = \{C \in \PP_L \ | \ C \preceq X \mbox{ or } Y \preceq C\}. \]

Geometrically, $\PP_L^+$ consists of $X$, $Y$, and those circles in $\PP_L$ that are in the bounded interstice for $X$ and $Y$, while $\PP_L^-$ consists of $X$, $Y$, and those circles in $\PP_L$ that are in the unbounded interstice for $X$ and $Y$. Note that $\PP_L^+ \cup \PP_L^- = \PP_L$ and $\PP_L^+ \cap \PP_L^- = \{X,Y\}$.
\end{defn}

\begin{lem} \label{bounded-interstice}
For each circle $C \in \PP_L^+$, $C \not \in \{X,Y\}$, there exist circles $A$ and $B$ in $\PP_L^+$ of generation strictly less than that of $C$ such that $C$ fills the bounded interstice for $A$ and $B$. (Recall that the generation of a circle in a packing $\PP$ depends on a choice of generating triple; in this case, $\PP^{(0)} = \{X,Y,L\}$.)
\end{lem}
\begin{proof}

We work by induction on the generation of $C$. If $C$ is generation $1$, then it must be \emph{the} circle that fills the bounded interstice between $X$ and $Y$, so satisfies the lemma. Now suppose that $\gen(C) \geq 2$. Thinking about when the circle $C$ is added to the packing in the recursive construction of Definition \ref{recursive}, we see that there are exactly three mutually tangent circles of generation less than $C$ that are tangent to $C$. One of these circles must be the line $L$, so let $A$ and $B$ be the other two. It follows from the remark after Definition \ref{defn:pairinterstice} that $A,B \in \PP_L^+$. Now $C$ fills one of the interstices formed by $A$ and $B$. We need to show that it fills the bounded interstice.

Now exactly one of $A$ and $B$ must be of generation exactly one less than $C$. (To see this, we recall the procedure for recursively building an Apollonian packing. This procedure implies that as soon as $A$ and $B$ have been added, the circle $C$ will be added in the very next generation. On the other hand, no two circles of the same generation are tangent since they fill different interstices.) Suppose this is $B$, so that we have $\gen(C) > \gen(B) \ge 1$. By the induction hypothesis, $B$ fills the bounded interstice formed by two other circles of generation less than it. One of those must be $A$ and let the other be $D$. But now we see that $D$ fills the unbounded interstice for $A$ and $B$. Since $\gen(D) < \gen(B) < \gen(C)$, we cannot have $D = C$. It follows then that $C$ must fill the bounded interstice for $A$ and $B$.
\end{proof}

Our main tool for keeping track of the circles in a half-plane packing $\PP$ is a labelling for each circle in $\PP_L$ by a pair of integers $(a,b)$. The remainder of this section is devoted to the construction and properties of this labelling. In \S3 we relate this labelling to the curvatures of the circles in $\PP$ and use it to deduce information about similarities between different packings.

\begin{defn} \label{labelling}
We define a labelling function $\vec{x} = (x,y): \PP_L \to \mathbb{Z}^2$. We define the labelling recursively starting with $\vec{x}(X) = (1,0)$ and $\vec{x}(Y) = (0,1)$. For the remaining circles in $\PP_L$, the label is determined by the following rule:
\begin{center}
  If $C$ is the circle that fills the bounded interstice for $A$ and $B$, then
  \begin{equation}\label{eqn:labeldefn}
	\vec{x}(C) = \vec{x}(A) + \vec{x}(B).
  \end{equation}
\end{center}
At each stage of the construction of the packing from its generators, $X$, $Y$, and $L$, a new circle in $\PP_L$ fills either the bounded or unbounded interstice for a pair of circles already present. The equation above determines a label for each such new circle --- see Figure \ref{fig:labels}.
\end{defn}

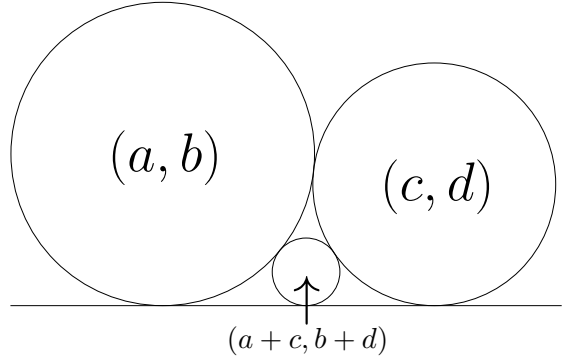
\begin{figure}[h]
	\begin{overpic}[scale = .9]{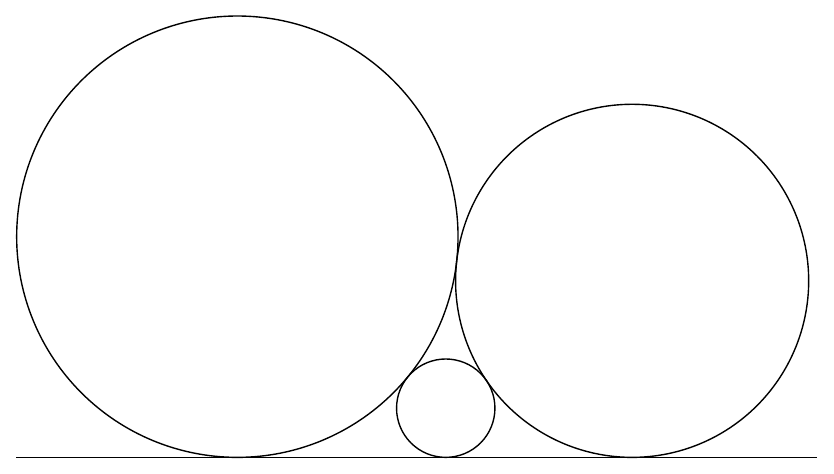}
	\put(41,55){\huge $(a,b)$}
	\put(144,44){\huge $(c,d)$}
	\put(86,-12){$(a+c,b+d)$}
	\put(111,1){\huge $\uparrow$}
	\end{overpic}
	
\vspace{5mm}
	
	\caption{The relationship satisfied by the labels.}
  \label{fig:labels}
\end{figure}

The main result of this section, Proposition \ref{uniquecircle}, tells us that the labelling function $\vec{x}$ is one-to-one and that for each pair $(a,b)$ of coprime integers, exactly one of $(a,b)$ and $(-a,-b)$ is in the image of $\vec{x}$. It also gives us a necessary and sufficient condition on the labels for two circles in $\PP_L$ to be tangent. It is convenient to start with this condition, which is stated in terms of the matrix formed by the labels of the two circles.

\begin{lem}\label{det1}
Let $A$ and $B$ be a pair of tangent circles in $\PP_L$ such that $A \prec B$. Then
\[ \begin{vmatrix}
x(A) & y(A)\\
x(B) & y(B)
\end{vmatrix}
= 1. \]
\end{lem}
\begin{proof}
The proof will be by induction on $\gen\{A,B\} := \max\{\gen(A),\gen(B)\}$. The base case is immediate: the generation zero circles form the pair $\{X,Y\}$, which are labeled $(1,0)$ and $(0,1)$ respectively. The corresponding matrix is the identity, which has determinant 1.

Now suppose $\gen\{A,B\} = n \ge 1$. First, observe that we cannot have $\gen(A) = \gen(B) = n$: if $\gen(A) = \gen(B) = n \ge 1$, then $A$ and $B$ were constructed to fill two \emph{disjoint} interstices in $\PP^{(n-1)}$ and cannot therefore be tangent. Hence $\{A,B\}$ contains a unique circle of generation $n$. Furthermore, because a generation $n$ circle is constructed to fill a single interstice in $\PP^{(n-1)}$, it is necessarily tangent to exactly three circles of generation strictly less than $n$. Therefore the circle of generation $n$ (either $A$ or $B$) is tangent to $L$ (generation zero), the circle in $\{A,B\}$ of smaller generation, and a third circle $C$ of generation strictly less than $n$.

There are three possibilities for the position of $C$ relative to $A$ and $B$: $C$ can be to the left of both, to the right of both, or between the two. Moreover, the generation $n$ circle can be either $A$ or $B$, so there are a total of six cases to consider. We only give the proof in two cases --- the other four are nearly identical. To prove them, we use the fact that the matrix row operations of row addition/subtraction are determinant-preserving and that switching two rows switches the sign of the determinant. In each case, the final equality holds by the induction hypothesis.

\textbf{Case 1:} Suppose $C \prec A \prec B$ and $\gen(B) = n$. Then $\vec{x}(A) = \vec{x}(C) + \vec{x}(B)$, and
\begin{equation*}
\begin{vmatrix}
\vec{x}(A) \\
\vec{x}(B)
\end{vmatrix}
=
\begin{vmatrix}
\vec{x}(A)\\
\vec{x}(B)-\vec{x}(A)
\end{vmatrix}
=
\begin{vmatrix}
\vec{x}(A) \\
-\vec{x}(C)
\end{vmatrix}
= -
\begin{vmatrix}
-\vec{x}(C)\\
\vec{x}(A)
\end{vmatrix}
=
\begin{vmatrix}
\vec{x}(C)\\
\vec{x}(A)
\end{vmatrix}
= 1.
\end{equation*}

\textbf{Case 2:} Suppose $A \prec C \prec B$ and $\gen(A) = n$. Then $\vec{x}(C) = \vec{x}(A) + \vec{x}(B)$, and
\begin{equation*}
\begin{vmatrix}
\vec{x}(A) \\
\vec{x}(B)
\end{vmatrix}
=
\begin{vmatrix}
\vec{x}(A)+\vec{x}(B) \\
\vec{x}(B)
\end{vmatrix}
=
\begin{vmatrix}
\vec{x}(C) \\
\vec{x}(B)
\end{vmatrix}
= 1.
\end{equation*}
\end{proof}

\begin{cor}\label{gcd1}
For any circle $C \in \PP_L$, $\gcd(x(C),y(C)) = 1$. \qed
\end{cor}

We now begin the proof that our labelling function $\vec{x}$ is one-to-one. We do this first for those circles in the bounded interstice for $X$ and $Y$.

\begin{lem}\label{boundedpositive}
For $C \in \PP_L^+$ we have $x(C),y(C) \geq 0$ with equality only if either $C = X$ or $C = Y$.
\end{lem}
\begin{proof}
This follows from (\ref{eqn:labeldefn}) by induction on generation since, by Lemma \ref{bounded-interstice}, the circle $C$ fills the bounded interstice of two circles of strictly smaller generation than it.
\end{proof}

Corollary \ref{gcd1} and Lemma \ref{boundedpositive} tell us that every circle in $\PP_L^+$ is labeled by a pair of nonnegative coprime integers. We now prove that every such pair is the label of a unique circle in $\PP_L^+$. At the same time, we prove the converse of Lemma \ref{det1} for $\PP_L^+$ --- that if circles $A,B \in \PP_L^+$ have the determinant of the matrix formed by their labels equal to $1$, then they are tangent with $A \prec B$. We first need the following elementary lemma.

\begin{lem}\label{technical}
Let $a$ and $b$ be positive, coprime integers. Then there exist unique integers $u$ and $v$ that satisfy the following properties:
\begin{enumerate}
\item $au - bv = 1$,
\item $0 < u \le b$, and
\item $0 \le v < a$.
\end{enumerate}
\end{lem}
\begin{proof}
Because $a,b$ are coprime, we can find an integer solution $(x,y)$ to the equation
\begin{equation}\label{eqn:elementary}
ax - by = 1.
\end{equation}
Given a particular solution $(x_0,y_0)$ to \eqref{eqn:elementary}, the entire solution set is
\[ \{(x,y) = (x_0 + kb, y_0 + ka) \ : \ k \in \zz\}. \]
There is then a \emph{unique} $k \in \zz$ such that $0 < x_0 + kb \le b$. Let $u := x_0 + kb$. Then $u$ satisfies property (ii). Setting $v := y_0 + ka$, property (i) is also satisfied, and property (iii) is a consequence of properties (i) and (ii).
\end{proof}

\begin{lem}\label{uniquecircle+}
Let
$\begin{bmatrix}
a & b\\
c & d
\end{bmatrix}$
be a determinant 1 matrix with nonnegative integer coefficients. Then there exist unique circles $C,C' \in \PP_L^+$ such that $\vec{x}(C) = (a,b)$ and $\vec{x}(C') = (c,d)$. Moreover, $C$ and $C'$ are tangent with $C \prec C'$.
\end{lem}
\begin{proof}
The proof is by induction on $\max\{a+b,c+d\}$. If $\max\{a+b,c+d\} = 1$, then necessarily
$\begin{bmatrix}
a & b\\
c & d
\end{bmatrix}
=
\begin{bmatrix}
1 & 0\\
0 & 1
\end{bmatrix}$. By Lemma \ref{boundedpositive}, all circles in $\PP_L^+$ different from $X$ and $Y$ must have $a + b > 1$, so there can be no circles in $\PP_L^+$, other than $X$ and $Y$, labeled by the pairs $(1,0)$ and $(0,1)$.

Once we have proved that there is a unique circle with label $(a,b)$, we denote that circle by $C_{(a,b)}$. At this point, therefore, we can write $C_{(1,0)} = X$ and $C_{(0,1)} = Y$.

Now fix an integer $n > 1$ and suppose we have proved the lemma, and hence constructed the circles $C_{(a,b)}$ and $C_{(c,d)}$, for any $a,b,c,d$ as in the statement of the lemma with $\max\{a+b,c+d\} < n$. We then take $a,b,c,d$ with $\max\{a+b,c+d\} = n$.

First of all, if $a+b = c+d$, then
\[ (d-b)(a+b) = d(a+b) - b(a+b) = d(a+b) - b(c+d) = ad-bc = 1, \]
so $a+b = 1$. Since $c+d = a+b$, this contradicts the fact that $\max\{a+b,c+d\} > 1$. Therefore $a+b > c+d$ or $a+b < c+d$. We prove the lemma in the case where $a+b > c+d$; the proof of the other case is virtually identical.

Since $a+b > c+d \ge 1$, we have that $0 < (a-c) + (b-d) < a + b$ and therefore
\[\max\{(a-c) + (b-d),c + d\} < a+b = n.\]
Since
$\begin{vmatrix}
a - c & b - d\\
c & d
\end{vmatrix}
= 1$, we may apply the induction hypothesis to the matrix
$\begin{bmatrix}
a-c & b-d\\
c & d
\end{bmatrix}$
once we show that $a-c \ge 0$ and $b-d \ge 0$.

Since $a + b > c + d$, we must have $a > c$ or $b > d$. If $a > c > 0$, then
\begin{align*}
ad - bc = 1 &\implies cd < bc + 1\\
&\implies cd \le bc\\
&\implies d \le b.
\end{align*}
(Note that if $c = 0$, then $ad = 1$, so $a = d = 1$. It follows from the fact that $a + b > c + d$ that $b \ge 1 = d$.) A similar argument shows that if we assume instead that $b > d$, then also $a > c$.

Therefore we have both $a - c \ge 0$ and $b - d \ge 0$, so the induction hypothesis tells us that there are unique circles $C_{(a-c,b-d)}$ and $\Ccd$ in $\PP_L^+$ satisfying $\vec{x}\left(C_{(a-c,b-d)}\right) = (a-c,b-d)$ and $\vec{x}\left(\Ccd\right) = (c,d)$ and that, moreover, these circles are tangent with $C_{(a-c,b-d)} \prec \Ccd$.

Now let $C$ be the circle that fills the bounded interstice for $C_{(a-c,b-d)}$ and $\Ccd$. Then
\begin{equation*}
\vec{x}(C) = (a-c,b-d) + (c,d) = (a,b),
\end{equation*}
so there exists a circle $C$ labeled by the pair $(a,b)$, which by construction is to the left of and tangent to $\Ccd$.

Finally, we must show that $C$ is the \emph{only} circle in $\PP_L^+$ that satisfies $\vec{x}(C) = (a,b)$. Suppose $C''$ is a circle with $\vec{x}(C'') = (a,b)$. By Lemma~\ref{bounded-interstice}, $C''$ fills the bounded interstice for two circles $A \prec B$ in $\PP_L^+$. Then $\vec{x}(C'') = \vec{x}(A) + \vec{x}(B)$, so we can write $\vec{x}(B) = (v,u)$ and $\vec{x}(A) = (a-v,b-u)$. Since $A$ and $B$ are both in $\PP_L^+$, Lemma \ref{boundedpositive} tells us that each of $v$, $u$, $a-v$, and $b-u$ is nonnegative, so $0 \le u \le b$ and $0 \le v \le a$.

In fact, we have $u \ne 0$ and $v \ne a$. Indeed, we know from Lemma \ref{boundedpositive} that if $u = 0$, then $B = X$. However, $X \preceq C'' \prec B$, so $X \ne B$, which means we cannot have $u = 0$. A similar argument shows that $a-v \ne 0$, and so $v \ne a$. Therefore $0 < u \le b$ and $0 \le v < a$, which are precisely properties (ii) and (iii) from Lemma \ref{technical}. That property (i) is satisfied follows by Lemma \ref{det1}, since $C''$ is tangent to and to the left of $B$. Since these three properties uniquely determine $(v,u)$, and since $(c,d)$ satisfies these three conditions by construction, we conclude that $v = c$ and $u = d$. By the uniqueness of $C_{(a-c,b-d)}$ and $\Ccd$ following from the induction hypothesis, we have $A = C_{(a-c,b-d)}$ and $B = \Ccd$, and it follows therefore that $C'' = C$.
\end{proof}
\begin{defn}
Lemma \ref{uniquecircle+} implies that any pair $(a,b)$ of coprime nonnegative integers is the label of a unique circle in $\PP_L^+$. As in the proof of Lemma~\ref{uniquecircle+}, we denote that circle by $C_{(a,b)}$.
\end{defn}

Lemma~\ref{uniquecircle+} yields a complete understanding of the labels of circles in $\PP_L^+$. We now define an operation which maps $\PP_L^+$ bijectively onto $\PP_L^-$, and use this to relate the labels of circles in $\PP_L^-$ to those of circles in $\PP_L^+$.

Let $\mathcal{I}$ be the unique circle which contains the three points of tangency among $C_{(1,0)}$, $C_{(0,1)}$, and $L$. (See Figure~\ref{fig:inversion}.) Define the map
\[\iota: \rr^2 \cup \{\infty\} \to \rr^2 \cup \{\infty\}\]
to be inversion with respect to $\mathcal{I}$.

\begin{figure}[h]
	\begin{overpic}[scale = .75]{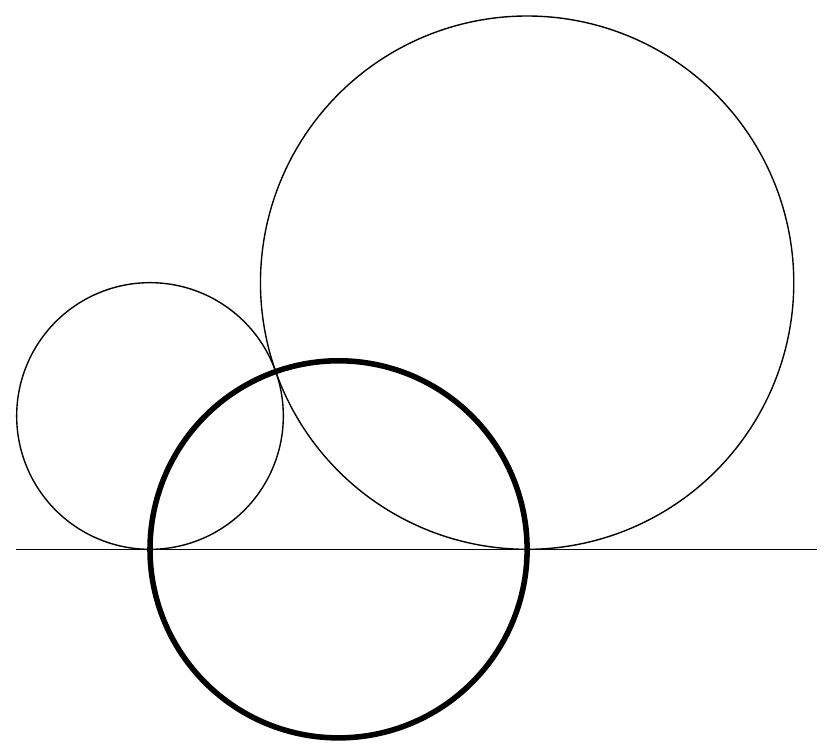}
	\put(20,85){$C_{(1,0)}$}
	\put(102,137){$C_{(0,1)}$}
	\put(72,8){$\mathcal{I}$}
	\put(170,35){$L$}
	\end{overpic}
	\caption{The inversion circle $\mathcal{I}$}
  \label{fig:inversion}
\end{figure}
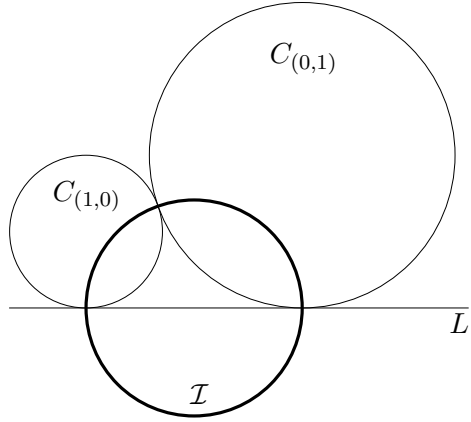

Let us pause to mention some of the relevant properties of $\mathcal{I}$ and the map $\iota$. Note that when we say that $\iota$ \emph{fixes} a particular circle or set of circles, we mean only as sets in $\rr^2$, not pointwise.
\begin{enumerate}
\item Inversion with respect to a circle is a bijection of order two; i.e., $\iota \circ \iota = \id$.
\item Since inversion maps circles to circles, $\iota$ maps Apollonian packings to Apollonian packings.
\item The inversion circle $\mathcal{I}$ intersects each of $C_{(1,0)}$, $C_{(0,1)}$, and $L$ orthogonally, and therefore $\iota$ fixes each of these three circles. Therefore, by property (ii) and Corollary~\ref{samepack}, $\iota$ fixes $\PP$ (and hence $\PP_L$ since $L$ is fixed).
\item The interior of $\mathcal{I}$ contains the bounded interstice for $C_{(1,0)}$ and $C_{(0,1)}$, and the exterior of $\mathcal{I}$ contains the unbounded interstice for $C_{(1,0)}$ and $C_{(0,1)}$. Since $\iota$ maps the interior of $\mathcal{I}$ to the exterior of $\mathcal{I}$, and vice versa, and since property (iii) holds, it follows that $\iota$ maps $\PP_L^+$ to $\PP_L^-$ and vice versa.
\end{enumerate}

Because $\mathcal{I}$ intersects $L$ orthogonally, the center of $\mathcal{I}$ lies on $L$. It makes sense, then, to talk about a circle $C \in \PP_L$ lying to the left or right of $\mathcal{I}$, by which we mean that the point of tangency of $C$ with $L$ lies to the left or right of the \emph{center} of $\mathcal{I}$. We now record two more properties of $\mathcal{I}$ and $\iota$:

\begin{enumerate}
\item[(v)] If $C$ lies to the left (resp. right) of $\mathcal{I}$, then $\iota(C)$ also lies to the left (resp. right) of $\mathcal{I}$. Furthermore, if $C \prec C'$ both lie to the left (resp. right) of $\mathcal{I}$, then $\iota(C') \prec \iota(C)$ both lie to the left (resp. right) of $\mathcal{I}$.
\item[(vi)] A circle $C$ contains the center of $\mathcal{I}$ (that is, the point of tangency between $C$ and $L$ is precisely the center of $\mathcal{I}$) if and only if $\iota(C)$ is a line parallel to $L$.
\end{enumerate}

As mentioned above, the reason for introducing the inversion map $\iota$ is to set up a one-to-one correspondence between $\PP_L^+$ and $\PP_L^-$. The following lemma establishes the connection between the labels of circles in $\PP_L^-$ and their images under $\iota$, which lie in $\PP_L^+$.

\begin{lem}\label{inversion}
Let $C \in \PP_L^-$, and let $C_{(a,b)} = \iota(C) \in \PP_L^+$ be the image of $C$ under the map $\iota$. Then
\begin{equation*}
\vec{x}(C) =
\begin{cases}
(a,-b) & \mbox{ if } C \preceq C_{(1,0)}\\
(-a,b) & \mbox{ if } C_{(0,1)} \preceq C
\end{cases}
.
\end{equation*}
\end{lem}
\begin{proof}
We will prove the statement by induction on the generation of $\iota(C)$. In the case that $\gen(\iota(C)) = 0$, we have $\iota(C) = C_{(1,0)}$ or $\iota(C) = C_{(0,1)}$. Since $C_{(1,0)}$ and $C_{(0,1)}$ are fixed by $\iota$ (and since $\iota$ is one-to-one), it follows that $C = C_{(1,0)}$ or $C = C_{(0,1)}$, and in both cases the statement holds.

Now suppose $\gen(\iota(C)) \ge 1$. Then, since $\Cab = \iota(C)$ lies in $\PP_L^+$ and is not equal to $C_{(1,0)}$ or $C_{(0,1)}$, Lemma \ref{bounded-interstice} tells us that $\Cab$ fills the bounded interstice for two circles $C_{(a_1,b_1)},C_{(a_2,b_2)} \in \PP_L^+$ with $\gen(C_{(a_1,b_1)}),\gen(C_{(a_2,b_2)}) < \gen(\Cab)$. Assume that $C_{(a_1,b_1)} \prec C_{(a_2,b_2)}$. By definition, $(a_1,b_1) + (a_2,b_2) = (a,b)$. Because inversion preserves tangencies, the circles $A = \iota(C_{(a_1,b_1)})$ and $B = \iota(C_{(a_2,b_2)})$ are tangent to each other as well as to $C$ and $L$.

There are a total of seven cases to consider, each corresponding to the position of the center of $\mathcal{I}$ with respect to the points of tangency of the circles $\Cab$, $C_{(a_1,b_1)}$, and $C_{(a_2,b_2)}$ with the line $L$ --- the center of $\mathcal{I}$ could lie on one of the three circles, it could lie between two of the circles, or it could lie to the left or right of all three of the circles. We will prove the result for two example cases; the proofs in the other cases are quite similar.

\textbf{Case 1:} Suppose that the center of $\mathcal{I}$ lies between the points of tangency of $C_{(a_1,b_1)}$ and $\Cab$ with $L$, as shown in Figure~\ref{fig:invcenter1}. In particular, $C_{(a_1,b_1)}$ is to the left of $\mathcal{I}$ and $\Cab \prec C_{(a_2,b_2)}$ are to the right. By property (v) above, we may conclude that $A$ is to the left of $\mathcal{I}$ and $B \prec C$ are to the right of $\mathcal{I}$. Because all three of the image circles necessarily lie in $\PP_L^-$, it follows that $A \preceq C_{(1,0)}$ and $C_{(0,1)} \preceq B \prec C$.

By the induction hypothesis, we know that $\vec{x}(A) = (a_1,-b_1)$ and $\vec{x}(B) = (-a_2,b_2)$. Since $A \prec B \prec C$, it follows that $\vec{x}(A) + \vec{x}(C) = \vec{x}(B)$, and so
\[\vec{x}(C) = \big(-(a_1 + a_2),b_1 + b_2\big) = (-a,b).\]

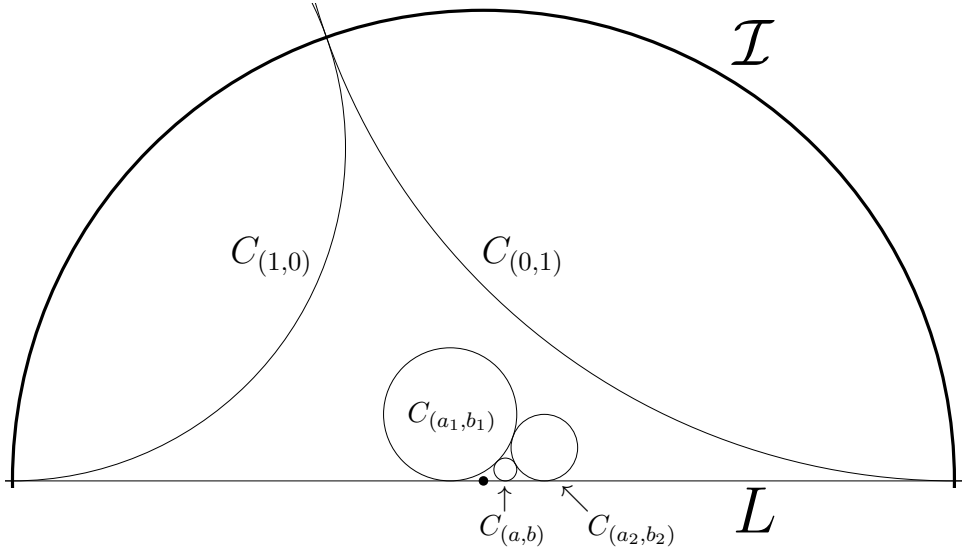
\begin{figure}[h]
	\begin{overpic}[scale = .75]{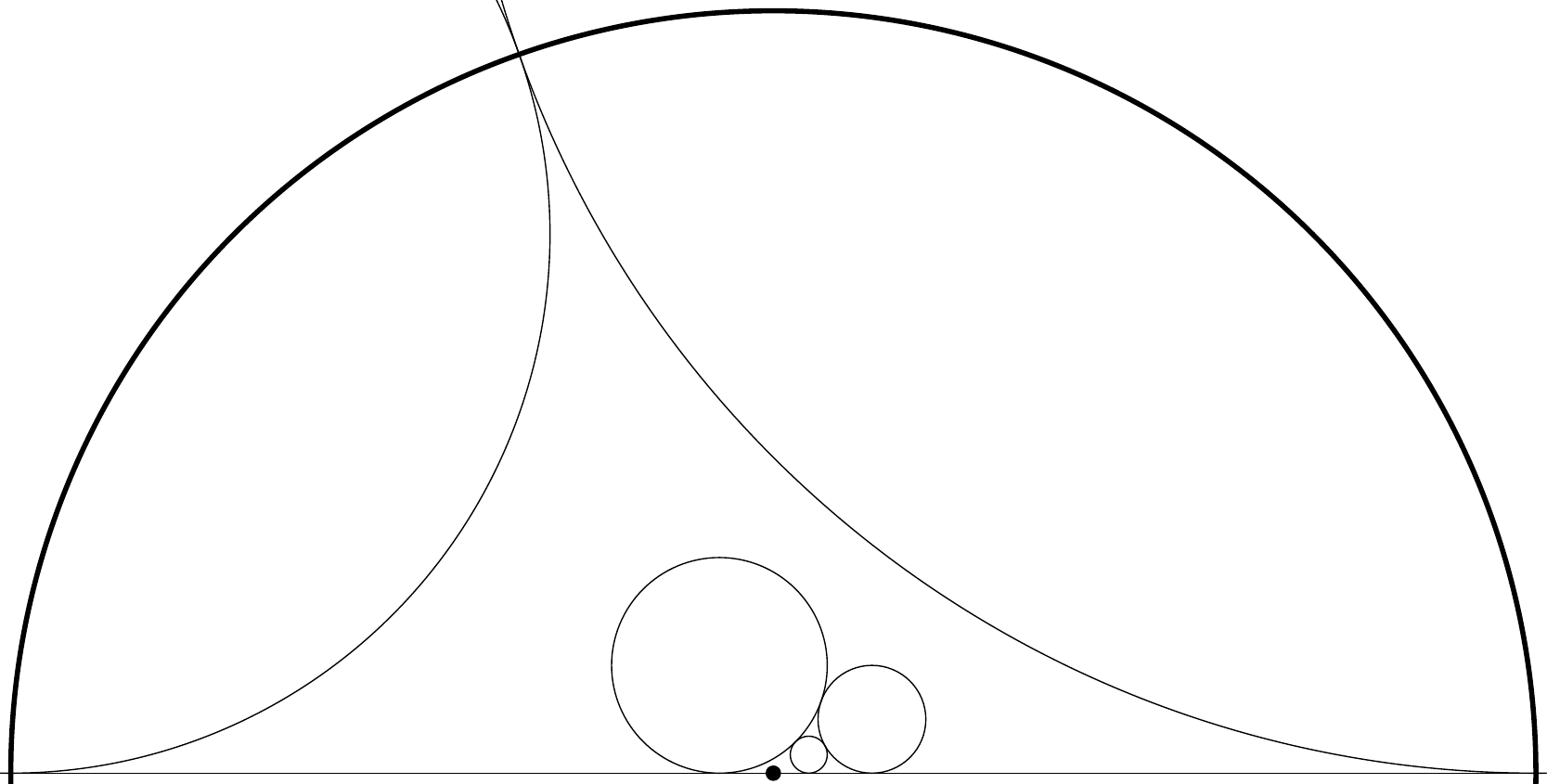}
	\put(85,85){\Large $C_{(1,0)}$}
	\put(180,85){\Large $C_{(0,1)}$}
	\put(275,160){\Huge $\mathcal{I}$}
	\put(186,-7){$\uparrow$}
	\put(179,-18){$\Cab$}
	\put(152,25){$C_{(a_1,b_1)}$}
	\put(210,-7){$\nwarrow$}
	\put(220,-18){$C_{(a_2,b_2)}$}
	\put(275,-17){\Huge $L$}
	\end{overpic}
	
	\vspace{5mm}
	
	\caption{Case 1 of Lemma \ref{inversion}. The point on the line $L$ represents the center of $\mathcal{I}$.}
  \label{fig:invcenter1}
\end{figure}

\textbf{Case 2:} Suppose that the center of $\mathcal{I}$ coincides with the point of tangency between $\Cab$ and $L$, as shown in Figure~\ref{fig:invcenter2}. Then $A \preceq C_{(1,0)}$, $C_{(0,1)} \preceq B$, and $C$ is a line parallel to $L$, which means that $C$ is to the left of every circle in $\PP_L$. By induction, we have $\vec{x}(A) = (a_1,-b_1)$ and $\vec{x}(B) = (-a_2,b_2)$. Since $C \prec A \prec B$, we have $\vec{x}(A) = \vec{x}(C) + \vec{x}(B)$. It follows that
\[\vec{x}(C) = \big(a_1 + a_2, -(b_1 + b_2)\big) = (a,-b).\]

\begin{figure}[h]
	\begin{overpic}[scale = .75]{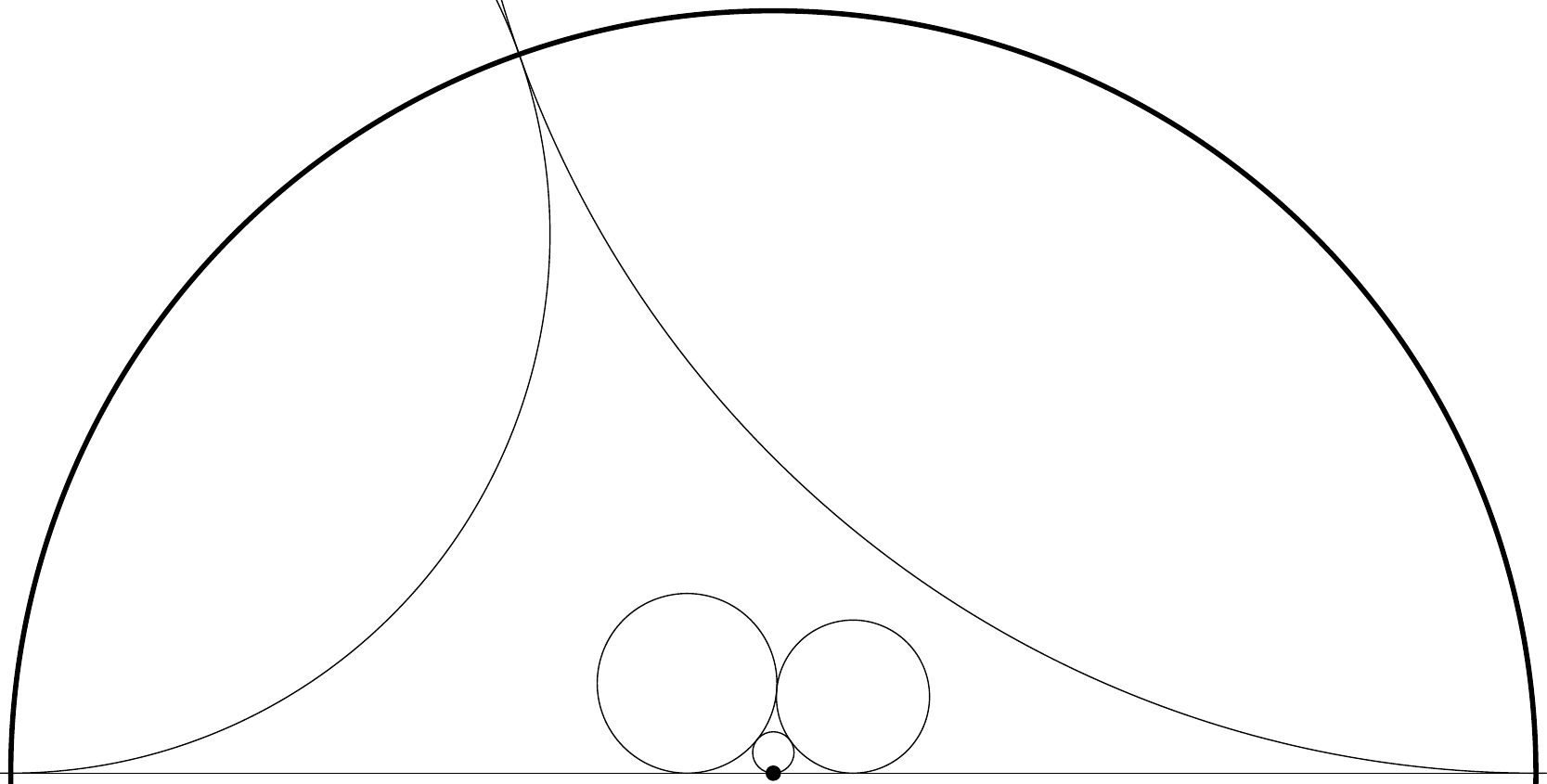}
	\put(85,85){\Large $C_{(1,0)}$}
	\put(180,85){\Large $C_{(0,1)}$}
	\put(275,160){\Huge $\mathcal{I}$}
	\put(178,-9){$\uparrow$}
	\put(170,-19){$\Cab$}
	\put(144,21){$C_{(a_1,b_1)}$}
	\put(183,19){\small $C_{(a_2,b_2)}$}
	\put(275,-17){\Huge $L$}
	\end{overpic}
	
	\vspace{5mm}
	
	\caption{Case 2 of Lemma \ref{inversion}}
  \label{fig:invcenter2}
\end{figure}
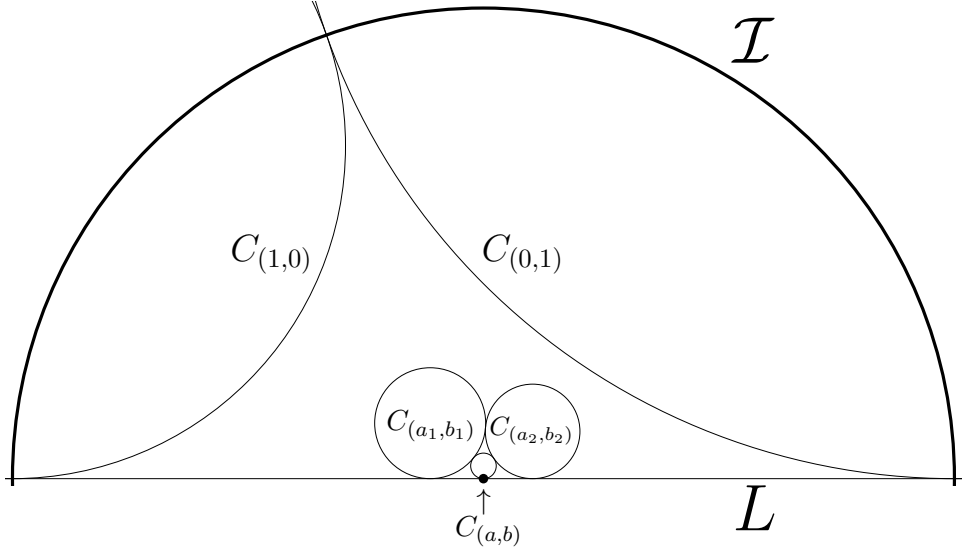

As mentioned above, the proofs of the other five cases are very similar to these two.
\end{proof}

Combining Lemmas \ref{uniquecircle+} and \ref{inversion} we obtain a complete understanding of how the circles in $\PP_L$ are labeled and when two labeled circles are tangent to one another.

\begin{prop}\label{uniquecircle}
For any integers $a$ and $b$ with $\gcd(a,b) = 1$, there is either
a unique circle in $\PP_L$ labelled by $(a,b)$ or a unique circle labelled
by $(-a,-b)$, but not both. If $a$ and $b$ are both nonnegative then the label is $(a,b)$. If $\Cab$ and $\Ccd$ are the unique circles in $\PP_L$ with labels $(a,b)$ and $(c,d)$ respectively, then $\Cab$ is tangent to $\Ccd$ on the
left if and only if
$\begin{vmatrix}
a & b\\
c & d
\end{vmatrix}
= 1$.
\end{prop}
\begin{proof}
Lemmas \ref{uniquecircle+} and \ref{inversion}, together with the fact that the inversion operation $\iota$ acts as a bijection between $\PP_L^+$ and $\PP_L^-$, imply the first claim. Lemma \ref{det1} already tells us the `only if' part of the second statement. So consider circles $C_{(a,b)}$ and $C_{(c,d)}$ with $ad - bc = 1$. If all $a,b,c,d$ are nonnegative then Lemma \ref{uniquecircle+} tells us $C_{(a,b)}$ and $C_{(c,d)}$ are tangent. If one of $a,b$ is negative, then the condition $ad - bc = 1$ implies that one of $c,d$ must be either negative or zero. But then Lemmas \ref{uniquecircle+} and \ref{inversion} imply that the circles $\iota(C_{(a,b)})$ and $\iota(C_{(c,d)})$ are tangent. Since $\iota$ preserves tangencies, it follows that $C_{(a,b)}$ and $C_{(c,d)}$ are also tangent. That $\Cab$ is to the left of $\Ccd$ follows from Lemma~\ref{det1}.
\end{proof}

\section{Self-similar half-plane packings}\label{sec:selfsim}

We are now at a point where we may begin to describe the similarities between two half-plane packings. First we recall exactly what is meant by a similarity of $\rr^2$.

\begin{defn}
The map $\Phi : \rr^2 \to \rr^2$ is called a \textbf{similarity of $\rr^2$} if there exists some constant $\mu > 0$ such that
\begin{equation*}
||\Phi(x) - \Phi(y)|| = \mu||x - y||
\end{equation*}
for all $x,y \in \rr^2$.
\end{defn}

Every similarity of the plane takes the form
\[ \Phi(x) = \mu A x + b, \]
where $\mu > 0$, $A$ is an orthogonal matrix, and $b \in \rr^2$. We say that $\Phi$ is \textbf{orientation-preserving} if $\det A = +1$, and  \textbf{orientation-reversing} if $\det A = -1$.

The set of similarities of $\rr^2$ forms a group under composition, called the \textbf{similarity group} of $\rr^2$, which we will denote by $\sss$. The orientation-preserving similarities form a subgroup $\sss_+$.

Similarities take circles to circles, and preserve tangency, so they take Apollonian circle packings to Apollonian circle packings. A key fact about the action of similarities on circle packings is the following.

\begin{lem}\label{3circ}
Let $\Phi$ be a similarity of $\rr^2$, and let $A$, $B$, and $C$ be three mutually tangent circles with disjoint interiors. If $A$, $B$, and $C$ have collinear centers, then $\Phi$ is determined by the three circles $\Phi(A)$, $\Phi(B)$, and $\Phi(C)$, up to a reflection in the line on which the centers of the image circles lie. If $A$, $B$, and $C$ have non-collinear centers, then $\Phi$ is completely determined by $\Phi(A)$, $\Phi(B)$, and $\Phi(C)$.
\end{lem}

\begin{rem}
Here we mean that the `center' of a line $L$ lies `at infinity' orthogonal to $L$ in the direction of its chosen interior. If one of the circles $A,B,C$, say $A$, is a line, then the collinearity condition is satisfied if and only if another of the circles, say $B$, is also a line, parallel to $A$. In this case $\Phi$ is determined up to a reflection in the line through the center of $C$ that is orthogonal to $A$ and $B$.
\end{rem}

\begin{proof}
Since similarities form a group, it suffices to consider the similarities that \emph{fix} $A$, $B$, and $C$ (as sets, not pointwise). If a similarity fixes the circles $A$, $B$ and $C$, then it fixes their centers. A similarity that fixes three non-collinear points must be the identity. One that fixes three distinct collinear points is either the identity or a reflection in the line formed by them.
\end{proof}

\begin{defn}
For $\Phi \in \sss$ and an Apollonian packing $\PP$, we write
\begin{equation*}
\Phi \cdot \PP := \{\Phi(C) : C \in \PP\}.
\end{equation*}
Two packings $\PP$ and $\PP'$ are \textbf{similar} if $\PP' = \Phi \cdot \PP$ for some $\Phi \in \sss$.

The group
\begin{equation*}
\Symm(\PP) := \{\Phi \in \sss \ | \ \Phi \cdot \PP = \PP\}
\end{equation*}
is the \textbf{self-similarity group of $\PP$}. The subgroup of $\Symm(\PP)$ consisting only of orientation-preserving similarities is the \textbf{orientation-preserving self-similarity group of $\PP$}, denoted by $\Symm_+(\PP)$. A packing $\PP$ is \textbf{self-similar} if $\Symm(\PP)$ is nontrivial.
\end{defn}

In order to establish similarity between two packings, we look at the curvatures of the circles involved.

\begin{defn}
For a circle $C$ in $\rr^2$, the \textbf{curvature} of $C$, denoted $\curv(C)$ is the reciprocal of the radius of $C$. A straight line in $\rr^2$ is considered to have curvature zero.
\end{defn}

Lemma \ref{3circ} allows us to check similarity by looking only at the curvatures in a triple of mutually tangent circles in each packing.

\begin{lem} \label{similar}
The packings $\PP$ and $\PP'$ are similar if and only if they contain triples of mutually tangent circles $(A,B,C)$ and $(A',B',C')$ respectively, such that there exists $\mu > 0$ with
\[ \curv(A') = \mu\curv(A), \; \curv(B') = \mu \curv(B), \; \curv(C') = \mu \curv(C). \]
\end{lem}
\begin{proof}
If $\PP$ and $\PP'$ are similar via similarity $\Phi$ with scale factor $\mu$, then take any triple $(A,B,C)$ and set $A' = \Phi(A), B' = \Phi(B), C' = \Phi(C)$. To prove the converse, choose a similarity $\Phi$ of the plane that takes $A$ to $A'$, $B$ to $B'$ and $C$ to $C'$. (One can choose a translation composed with dilation to get $A$ to $A'$, add a rotation to get $B$ to $B'$, then add a reflection if necessary to get $C$ to $C'$.) By construction, $\PP'$ and $\Phi(\PP)$ both contain the triple $\{A',B',C'\}$, and therefore $\PP' = \Phi(\PP)$ by Corollary \ref{samepack}.
\end{proof}

Turning now to half-plane packings, that is, those that have a straight line for at least one of the circles, recall that we can focus on the following packings.

\begin{defn}\label{palpha}
Let $\alpha \in \rr$.  Then $\PP_{\alpha}$ is the packing generated by a triple $\{X,Y,L\}$ of mutually tangent circles, where $L$ is the $x$-axis in $\rr^2$, $X$ is a circle of curvature $\alpha^2$ tangent to (and above) $L$ at the origin, and $Y$ is a circle of curvature 1 resting on $L$ and tangent (on the right) to $X$. The generating triple is illustrated in Figure \ref{P_alpha} in the Introduction.
\end{defn}

\begin{lem}\label{simtopalpha}
Every half-plane packing $\PP$ is similar, via an orientation-preserving similarity, to $\PP_{\alpha}$ for some $\alpha > 0$.
\end{lem}
\begin{proof}
Choose any two tangent circles in $\PP$ that are tangent to a line but are not themselves lines. Taking $\mu$ to be the ratio of their curvatures (in the appropriate order), and $\alpha = \sqrt{\mu}$, this follows from \ref{similar}.
\end{proof}

The key to analyzing Apollonian circle packings is the following result, due to Descartes. This describes the relationship between the curvatures of four mutually tangent circles in the plane. A selection of proofs of this are given in \cite{pedoe:1967}.

\begin{thm}[Descartes' Circle Theorem]\label{descartes}
Let $w$, $x$, $y$, and $z$ represent the curvatures of four mutually tangent circles in the Euclidean plane. Then
\begin{equation*}
2(w^2 + x^2 + y^2 + z^2) = (w + x + y + z)^2.
\end{equation*}
\end{thm}

For half-plane packings, we apply this Theorem in the case where one of the four circles is a line, i.e. has zero curvature. In this case, the quadratic relationship boils down to a linear relationship between the square roots of the curvatures of the circles.

\begin{cor}\label{descartesline}
Let $\alpha^2$, $\beta^2$ and $\gamma^2$ represent the curvatures of three mutually tangent circles all tangent to a line $L$, where $\alpha \ge \beta \ge 0$ and $\gamma \ge 0$. Then
\begin{equation}
\gamma = \alpha \pm \beta. \label{eqn1}
\end{equation}
In particular, when the circle of curvature $\gamma^2$ lies in the bounded interstice formed by the others, we have
\begin{equation}
\gamma = \alpha + \beta. \label{eqn2}
\end{equation}
\end{cor}

\begin{proof}
The proof of (\ref{eqn1}) follows from Theorem \ref{descartes} by setting $w = 0$ and applying the quadratic formula appropriately. The proof of (\ref{eqn2}) follows from the fact that the circle in the bounded interstice has a curvature at least as large as that of the two circles surrounding it.
\end{proof}

We illustrate (\ref{eqn2}) in Figure \ref{fig:descartes}. Observe that the illustration of this equation is virtually identical to the illustration in Figure \ref{fig:labels}, which shows the recursive labelling process defined in \S\ref{sec:acp}. It is this linear relationship between the curvatures of tangent circles that inspires that labelling. A key consequence of this connection is Lemma \ref{curv} below.

\begin{figure}[h]
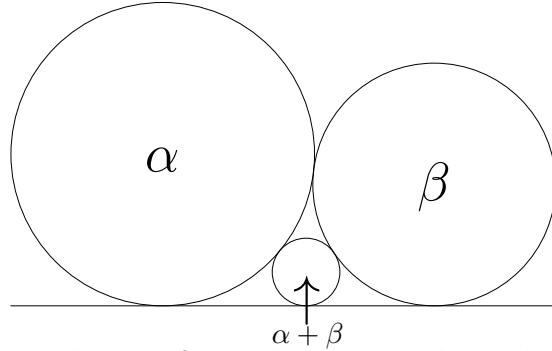

	\begin{overpic}[scale = .9]{triple_on_L.pdf}
	\put(55,55){\huge $\alpha$}
	\put(158,44){\huge $\beta$}
	\put(103,-10){$\alpha+\beta$}
	\put(111,1){\huge $\uparrow$}
	\put(0,-25){$\mbox{}$}
	\end{overpic}
	\caption{Descartes' Theorem for circles in $\PP_L$. The circles are labelled by the \emph{square roots} of their curvatures.}
 \label{fig:descartes}
\end{figure}

\begin{defn}
Consider now some fixed real number $\alpha > 0$ and recall the packing $\PP_\alpha$ from Definition \ref{palpha}. We write $\PP_{\alpha,L}$ for the set of circles in $\PP_{\alpha}$ that are tangent to the line $L$ (that is, the $x$-axis). Let $C_{(1,0)}$ denote the circle of curvature $\alpha^2$ that is tangent to $L$ at the origin, and let $C_{(0,1)}$ denote the circle of curvature $1$ that is tangent to $L$ and to $C_{(1,0)}$ on the right. As described in the previous section, these choices determine a unique label for each circle in $\PP_{\alpha,L}$. When we need to specify the underlying $\alpha$ we use a superscript, as in $C_{(a,b)}^{\alpha}$, but we often drop the $\alpha$ when context allows.
\end{defn}

\begin{lem}\label{curv}
The circle $\Cab$ in $\PP_{\alpha,L}$ has curvature given by
\[ \curv(C_{(a,b)}) = (a\alpha + b)^2. \]
Moreover, $a\alpha+b \geq 0$.
\end{lem}
\begin{proof}
The proof is by induction on the generation of $\Cab$ with respect to the generating triple $\{L,C_{(1,0)},C_{(0,1)}\}$. The result is immediately seen to hold for the generation zero circles $C_{(1,0)}$ and $C_{(0,1)}$, since they were chosen to satisfy $\curv(C_{(1,0)}) = \alpha^2$ and $\curv(C_{(0,1)}) = 1$.

Now suppose $\gen(\Cab) = n \ge 1$. The circle $\Cab$ was constructed to fill an interstice bounded by three circles of generation strictly less than $n$; since $\Cab$ is tangent to $L$, $L$ is necessarily one of those circles. Because the other two circles are tangent to $L$ as well, we can call them $C_{(a_1,b_1)}$ and $C_{(a_2,b_2)}$ with $C_{(a_1,b_1)} \prec C_{(a_2,b_2)}$. We do the case where $C_{(a,b)}$ fills the bounded interstice between $C_{(a_1,b_1)}$ and $C_{(a_2,b_2)}$. The case where it fills the unbounded interstice, either to the left or right, is similar.

By Definition \ref{labelling}, we have $a = a_1 + a_2$ and $b = b_1 + b_2$. By Corollary \ref{descartesline} then, we get
\begin{align*}
\sqrt{\curv(\Cab)} &= \sqrt{\curv(C_{(a_1,b_1)})} + \sqrt{\curv(C_{(a_2,b_2)})}\\
&= a_1\alpha + b_1 + a_2\alpha + b_2\\
&= a\alpha + b,
\end{align*}
where the second equality holds by induction.
\end{proof}

The following is an immediate consequence of Lemma \ref{curv}:

\begin{cor}\label{uniquecurv}
If $\alpha \not \in \qq$, then no two circles in $\PP_{\alpha}$ have the same curvature.
\end{cor}
\begin{proof}
Suppose $\curv(\Cab) = \curv(C_{(a',b')})$. Then, by Lemma \ref{curv}, we have $a\alpha + b = a'\alpha + b'$. Since $\alpha$ is not rational, the only way for this equation to hold is for $a = a'$ and $b = b'$ which, by the uniqueness statement in Proposition \ref{uniquecircle}, implies that $\Cab = C_{(a',b')}$.
\end{proof}

We are now in a position to prove our first main result, identifying the set of similarities between the two packings $\PP_\alpha$ and $\PP_\beta$ when $\alpha,\beta$ are positive real numbers. This is Theorem \ref{thm1-intro} from the Introduction. We start by showing how to associate a matrix to such a similarity.

\begin{defn} \label{sigma}
Fix $\alpha,\beta > 0$ and let $\Phi$ be a similarity of $\rr^2$ such that $\Phi \cdot \PP_\beta = \PP_\alpha$. Also assume that $\Phi(L) = L$, where $L$ is the $x$-axis, that is, the chosen line in each packing. Then $\Phi$ takes mutually tangent circles in $\PP_{\beta,L}$ to mutually tangent circles in $\PP_{\alpha,L}$. In particular, we have
\[ \Phi(C_{(1,0)}^{\beta}) = C_{(a,b)}^{\alpha}, \quad \Phi(C_{(0,1)}^{\beta}) = C_{(c,d)}^{\alpha} \]
for some integers $a,b,c,d$. By Lemma \ref{det1},
\[ \begin{bmatrix}
a & b\\
c & d
\end{bmatrix} \]
is an integer matrix of determinant $\pm 1$. The determinant is $+1$ if $\Cab^{\alpha} \prec \Ccd^{\alpha}$, in which case $\Phi$ is orientation-preserving, and $-1$ if $\Phi$ is orientation-reversing. We denote this matrix by $\mathbf{A}(\Phi)$.
\end{defn}

\begin{thm}\label{linfracequiv++}
Let $\alpha, \beta > 0$. The construction $\mathbf{A}$ of Definition \ref{sigma} determines a bijection between the set of similarities of $\rr^2$ that take $\PP_{\beta}$ to $\PP_{\alpha}$ (and fix the $x$-axis) and the set of matrices $\begin{bmatrix} a & b \\ c & d \end{bmatrix}$ in $\PGL_2(\mathbb{Z})$ that satisfy
\[\beta = \frac{a\alpha + b}{c\alpha + d}. \]
Furthermore, the restriction of $\mathbf{A}$ to orientation-preserving similarities is a bijection onto the set of elements of $\PSL_2(\zz) = \SL_2(\zz) \big / \{\pm 1\}$ with this property.
\end{thm}
\begin{proof}
We first show that $\mathbf{A}(\Phi)$ satisfies the condition that $\beta = \dfrac{a\alpha + b}{c\alpha + d}$. Because $\Phi$ is a similarity, there exists some $\lambda > 0$ such that $\curv(\Phi(C)) = \lambda \curv(C)$ for all $C \in \PP_{\beta}$. Since $\curv(C_{(1,0)}^{\beta}) = \beta^2$ and $\curv(C_{(0,1)}^{\beta}) = 1$, it follows that $\curv(\Cab^{\alpha}) = \lambda \beta^2$ and $\curv(\Ccd^{\alpha}) = \lambda$.
By taking square roots and applying Lemma \ref{curv}, we may conclude that
\begin{align*}
a\alpha + b &= \sqrt{\lambda} \beta\\
c\alpha + d &= \sqrt{\lambda},
\end{align*}
which we may rewrite as $\dfrac{a\alpha + b}{c\alpha + d} = \beta$.

To show that $\mathbf{A}$ is injective, suppose $\mathbf{A}(\Phi) = \mathbf{A}(\Phi')$ in $\PGL_2(\zz)$. Then, if $\Phi(C_{(1,0)}^{\beta}) = C_{(a,b)}^{\alpha}$ and $\Phi'(C_{(1,0)}^{\beta}) = C_{(a',b')}^{\alpha}$, we must have $(a,b) = \pm(a',b')$. But, by Proposition \ref{uniquecircle}, only one of $(a,b)$ and $(-a,-b)$ is the label of a circle in $\PP_\alpha$. Therefore, in fact $(a,b) = (a',b')$ and so $\Phi(C_{(1,0)}^{\beta}) = \Phi'(C_{(1,0)}^{\beta})$. Similarly $\Phi(C_{(0,1)}^{\beta}) = \Phi'(C_{(0,1)}^{\beta})$. Since also $\Phi(L) = \Phi'(L)$, Lemma \ref{3circ} tells us that $\Phi = \Phi'$.

Now let
$\begin{bmatrix}
a & b\\
c & d
\end{bmatrix}$
be an element of $\PGL_2(\zz)$ such that $\beta = \dfrac{a\alpha + b}{c\alpha + d}$. Because the determinant of this matrix is $\pm 1$, Proposition \ref{uniquecircle} tells us that either $(a,b)$ or $(-a,-b)$, but not both, is the label of a circle in $\PP_{\alpha,L}$, and that the same holds for $(\pm c, \pm d)$. Furthermore, these circles are tangent. Because we can multiply the matrix by $-1$ and not change it in $\PGL_2(\zz)$, we may assume that $(a,b)$ is the label for a circle in $\PP_{\alpha,L}$. Now we need to show that $(c,d)$ is also the label of a circle in $\PP_{\alpha,L}$. Suppose that $(-c,-d)$, rather than $(c,d)$, is a label in $\PP_{\alpha,L}$. Lemma \ref{curv} tells us that $a\alpha + b \ge 0$. This statement, along with the fact that $\dfrac{a\alpha + b}{c\alpha + d} = \beta > 0$, implies that $c\alpha + d > 0$, and therefore $(-c)\alpha + (-d) < 0$, which contradicts Lemma \ref{curv} applied to the circle $C_{(-c,-d)}^{\alpha}$. We may therefore conclude that $\Cab^{\alpha}$ and $\Ccd^{\alpha}$ form a pair of tangent circles in $\PP_{\alpha,L}$. Then, since the ratio of the curvatures of $\Cab^{\alpha}$ and $\Ccd^{\alpha}$ is
\[ \frac{(a\alpha + b)^2}{(c\alpha + d)^2} = \beta^2 \]
which is the same as the ratio of the curvatures of $C_{(1,0)}^{\beta}$ and $C_{(0,1)}^{\beta}$, Lemma \ref{similar} determines a similarity $\Phi$ between $\PP_{\alpha}$ and $\PP_{\beta}$ such that
\[ \mathbf{A}(\Phi) = \begin{bmatrix} a & b\\ c & d \end{bmatrix}. \]

This shows that $\mathbf{A}$ is a bijection. We have already noted that $\mathbf{A}(\Phi) \in \PSL_2(\zz)$ if and only if $\Phi$ is orientation-preserving, which gives us the last part of the Theorem.
\end{proof}

We have the following corollaries.

\begin{cor}\label{strip}
$\PP_{\alpha}$ is a strip packing if and only if $\alpha \in \qq^+$.
\end{cor}
\begin{proof}
First, we note that $\PP_{\alpha}$ is a strip packing if and only if it is similar to the packing $\PP_1$. By Theorem \ref{linfracequiv++}, this is true if and only if there is an integer matrix
$\begin{bmatrix}
a & b\\
c & d
\end{bmatrix}$
of determinant $\pm 1$ such that $\dfrac{a\cdot 1 + b}{c\cdot 1 + d} = \dfrac{a+b}{c+d} = \alpha$. Certainly, if such a matrix exists, then $\alpha$ is rational. Conversely, suppose $\alpha = \frac{p}{q}$ with $p,q > 0$ and $\gcd(p,q) = 1$. Let $a$ and $c$ be positive integers that satisfy $aq - cp = 1$, and set $b = p - a$, $d = q - c$. By construction,
$\begin{vmatrix}
a & b\\
c & d
\end{vmatrix}
= 1$ and $\dfrac{a+b}{c+d} = \dfrac{p}{q} = \alpha$.
\end{proof}

\begin{cor}\label{quad}
If $\PP_{\alpha}$ is self-similar, then $\alpha$ is the root of a quadratic polynomial with rational coefficients.
\end{cor}
\begin{proof}
If $\PP_{\alpha}$ is self-similar, then there is a nontrivial similarity $\Phi$ that maps $\PP_{\alpha}$ to itself. By Theorem \ref{linfracequiv++}, this corresponds to a nontrivial element
$\begin{bmatrix}
a & b\\
c & d
\end{bmatrix}
\in \PGL_2(\zz)$ such that $\dfrac{a\alpha + b}{c\alpha + d} = \alpha$; i.e., such that
\begin{equation*}
c\alpha^2 + (d-a)\alpha - b = 0.
\end{equation*}
It is easy to check that the only way for all three coefficients to be zero is for $a = d = \pm 1$, $b = c = 0$, which contradicts the fact that the matrix
$\begin{bmatrix}
a & b\\
c & d
\end{bmatrix}$
 is not the identity in $\PGL_2(\zz)$.
\end{proof}

The rest of this section is concerned with proving the converse of Corollary \ref{quad}: if $\alpha$ is the root of a quadratic polynomial with rational coefficients, then $\PP_{\alpha}$ is self-similar. This follows from Theorem \ref{thm} below, which is Theorem \ref{symmgroup} of the Introduction.

Theorem \ref{thm} goes beyond identifying which packings are self-similar. We in fact calculate the self-similarity groups of all the packings. To do this we first show that, in the case $\beta = \alpha$, the bijection of Theorem \ref{linfracequiv++} is a group isomorphism. This is the content of the following proposition.

\begin{prop}\label{isomstab}
Let $\alpha > 0$ be irrational. Then there are group isomorphisms
\begin{align*}
\Symm(\PP_{\alpha}) &\cong \Stab(\alpha) := \left\{\left.\begin{bmatrix} a & b \\ c & d \end{bmatrix} \in \PGL_2(\zz) \; \right| \; \frac{a\alpha + b}{c\alpha + d} = \alpha \right\} \mbox{, and}\\
\Symm_+(\PP_{\alpha}) &\cong \Stab_+(\alpha) := \left\{\left.\begin{bmatrix} a & b \\ c & d \end{bmatrix} \in \PSL_2(\zz) \; \right| \; \frac{a\alpha + b}{c\alpha + d} = \alpha \right\}.
\end{align*}
\end{prop}
\begin{proof}
Since $\alpha$ is not rational by assumption, Corollary \ref{strip} tell us that  $\PP_{\alpha}$ is not a strip packing, so the $x$-axis $L$ is the unique line in the packing $\PP_{\alpha}$. Therefore \emph{every} self-similarity of $\PP_{\alpha}$ maps $L$ to $L$. If we take $\beta = \alpha$, then Theorem \ref{linfracequiv++} states precisely that we have bijections of the above forms given by the construction $\mathbf{A}$. Now we show that when $\beta = \alpha$, these bijections are group isomorphisms.

Let $\Phi$ and $\Phi'$ be elements of $\Symm(\PP_{\alpha})$. Say
$$\mathbf{A}(\Phi) =
\begin{bmatrix}
a & b\\
c & d
\end{bmatrix}
\mbox{ \ \ and \ \ }
\mathbf{A}(\Phi') =
\begin{bmatrix}
a' & b'\\
c' & d'
\end{bmatrix}.$$
We must show that $\mathbf{A}(\Phi \circ \Phi') =
\begin{bmatrix}
a & b\\
c & d
\end{bmatrix}
\begin{bmatrix}
a' & b'\\
c' & d'
\end{bmatrix}
=
\begin{bmatrix}
aa' + bc' & ab' + bd'\\
ca' + dc' & cb' + dd'
\end{bmatrix}$.
By Corollary \ref{uniquecurv} and Lemma \ref{3circ}, it will suffice to show
\begin{align}
\curv ((\Phi \circ \Phi')(C_{(1,0)})) &= ((aa' + bc')\alpha + (ab' + bd'))^2 \mbox{ and } \label{eqn3}\\
\curv ((\Phi \circ \Phi')(C_{(0,1)})) &= ((ca' + dc')\alpha + (cb' + dd'))^2. \label{eqn4}
\end{align}
In fact, it will suffice to show only that (\ref{eqn4}) holds: since $\Phi \circ \Phi'$ is a similarity, we must have $\curv((\Phi \circ \Phi')(C_{(1,0)})) = \lambda\curv(C_{(1,0)}) = \lambda\alpha^2$ and $\curv((\Phi \circ \Phi')(C_{(0,1)})) = \lambda\curv(C_{(0,1)}) = \lambda$ for some $\lambda > 0$; therefore, if (\ref{eqn4}) is satisfied, the fact that the product of the matrices is still an element of $\Stab(\alpha)$ will force (\ref{eqn3}).

The scale factor $\lambda$ under the composition $\Phi \circ \Phi'$ is the product of the scale factors $\mu$ and $\mu'$ under the maps $\Phi$ and $\Phi'$ respectively. Since $\curv(C_{(0,1)}) = 1$, it follows that
\begin{align*}
\mu &= \curv(\Phi(C_{(0,1)})) = \curv(\Ccd) = (c\alpha + d)^2 \mbox{ and }\\
\mu' &= \curv(\Phi'(C_{(0,1)})) = \curv(C_{(c',d')}) = (c'\alpha + d')^2.
\end{align*}
Therefore we will have shown that $\mathbf{A}$ is a group homomorphism if we can show that
\begin{equation*}
(c\alpha + d)(c'\alpha + d') = (ca' + dc')\alpha + (cb' + dd').
\end{equation*}
Indeed,
\begin{align*}
\frac{a'\alpha + b'}{c'\alpha + d'} = \alpha &\implies c'\alpha^2 + (d' - a')\alpha - b' = 0\\
&\implies cc'\alpha^2 + (cd' - ca')\alpha - cb' = 0\\
&\implies cc'\alpha^2 + (cd' + dc' - ca' - dc')\alpha - cb' = 0\\
&\implies cc'\alpha^2 + (cd' + dc')\alpha = (ca' + dc')\alpha + cb'\\
&\implies cc'\alpha^2 + (cd' + dc')\alpha + dd' = (ca' + dc')\alpha + (cb' + dd')\\
&\implies (c\alpha + d)(c'\alpha + d') = (ca' + dc')\alpha + (cb' + dd').
\end{align*}
\end{proof}

Now we can prove the main result of this section.

\begin{thm} \label{thm}
Let $\alpha > 0$. Then:
\begin{enumerate}
  \item If $\alpha \in \mathbb{Q}$, then $\PP_{\alpha}$ is a strip packing and
  \[ \Symm(\PP_{\alpha}) \cong D_{\infty} \times \mathbb{Z}/2\mathbb{Z} \]
  with subgroup
  \[ \Symm_+(\PP_{\alpha})\cong D_{\infty}, \]
  where $D_{\infty}$ denotes the infinite dihedral group.
  \item If $\alpha$ is quadratic over $\mathbb{Q}$, then
  \[ \Symm(\PP_{\alpha}) \cong \mathbb{Z}. \]
  Let $D$ denote the discriminant of the primitive integral polynomial with root $\alpha$. Then the subgroup $\Symm_+(\PP_{\alpha})$ is:
  \begin{itemize}
    \item equal to $\Symm(\PP_{\alpha})$ if the Pell equation $x^2 - Dy^2 = -4$ has no integral solution for $(x,y)$;
    \item the index $2$ subgroup of $\Symm(\PP_{\alpha})$ if $x^2 - Dy^2 = -4$ does have an integral solution.
  \end{itemize}
  \item Otherwise
  \[ \Symm_+(\PP_{\alpha}) = \Symm(\PP_{\alpha}) = 1. \]
\end{enumerate}
\end{thm}
\begin{proof}
First of all, it follows immediately from Corollary \ref{quad} that, for any $\alpha > 0$ that is neither rational nor quadratic, $\PP_{\alpha}$ is not self-similar. Therefore $\Symm(\PP_{\alpha}) = \Symm_+(\PP_{\alpha}) = 1$ for all such $\alpha$.

Now suppose that $\alpha$ is rational. Then Corollary \ref{strip} tells us that $\PP_{\alpha}$ is a strip packing. In this case, the full self-similarity group is generated by a translation `along' the strip, a reflection in a line perpendicular to the strip, and a reflection that interchanges the two lines. The resulting group is isomorphic to $D_\infty \times \zz/2\zz$ where $D_\infty$ is the infinite dihedral group. The subgroup of orientation-preserving self-similarities of $\PP_{\alpha}$ is generated by the translation and the rotation given by combining the two reflections. This subgroup is isomorphic to $D_\infty$.

The main focus of our work is the case where $\alpha$ is of degree precisely $2$ over $\qq$. We have already shown in Proposition \ref{isomstab} that
\begin{equation*}
\Symm(\PP_{\alpha}) \cong \left\{\left.\begin{bmatrix} a & b \\ c & d \end{bmatrix} \in \PGL_2(\zz) \; \right| \; \frac{a\alpha + b}{c\alpha + d} = \alpha \right\}
\end{equation*}
and that
\begin{equation*}
\Symm_+(\PP_{\alpha}) \cong \left\{\left.\begin{bmatrix} a & b \\ c & d \end{bmatrix} \in \PSL_2(\zz) \; \right| \; \frac{a\alpha + b}{c\alpha + d} = \alpha \right\},
\end{equation*}
so it suffices to calculate these stabilizer groups.

It turns out that the elements of $\PGL_2(\zz)$ that fix $\alpha$ are closely related to the solutions to the Pell equations
\begin{equation}\label{eqn:pell}
x^2 - Dy^2 = \pm 4,
\end{equation}
where $D = q^2 - 4pr$ is the discriminant of the primitive integer polynomial $f(x) = px^2 + qx + r$ satisfied by $\alpha$, with $p > 0$.

Define the sets $\GG$ and $\GG_+$ as follows:
\begin{equation*}
\GG := \left\{\left. \frac{x + y\sqrt{D}}{2}\ \ \right| \ x,y \in \zz, \; x^2 - Dy^2 = \pm 4\right\}
\end{equation*}
and
\begin{equation*}
\GG_+ := \left\{\left. \frac{x + y\sqrt{D}}{2}\ \ \right| \ x,y \in \zz, \; x^2 - Dy^2 = 4\right\},
\end{equation*}
where $D$ is as in the previous paragraph. One can easily check that $\GG$ is a group under multiplication with subgroup $\GG_+$. Furthermore, one can show (see \cite[Theorem 1.9]{jacobson:2009}, for example) that
\begin{equation}
\GG \cong \mathbb{Z} \times \{\pm 1\}.
\label{eqn:pellgroup}
\end{equation}

Our proof of Theorem \ref{thm} is given by relating the stabilizer of $\alpha$ in $\PGL_2(\zz)$ with the group $\GG$.

Recall that the integers $p,q,r$ are the coefficients of the primitive integer polynomial satisfied by $\alpha$ and that $D = q^2 - 4pr$. We first construct a group homomorphism
\begin{align*}
\Gamma : \GG &\to \left\{\left.\begin{bmatrix} a & b \\ c & d \end{bmatrix} \in \PGL_2(\zz) \; \right| \; \frac{a\alpha + b}{c\alpha + d} = \alpha \right\} \\
\frac{x + y\sqrt{D}}{2} &\mapsto
\begin{bmatrix}
\frac{x-yq}{2} & -yr\\
yp & \frac{x+yq}{2}
\end{bmatrix}
\end{align*}
The matrices in the image of $\Gamma$ consist of integer entries because:
\begin{align*}
x \pm yq &\equiv x^2 - y^2q^2 \pmod 2\\
&\equiv x^2 - y^2(q^2 - 4pr) \pmod 2\\
&= x^2 - Dy^2 = \pm 4 \equiv 0 \pmod 2.
\end{align*}
These matrices stabilize $\alpha$ because:
\begin{align*}
p\alpha^2 + q\alpha + r = 0 &\implies yp\alpha^2 + yq\alpha + yr = 0\\
&\implies yp\alpha^2 + \left(\frac{x + yq}{2} - \frac{x - yq}{2}\right) \alpha + yr = 0\\
&\implies \frac{\frac{x-yq}{2} \alpha - yr}{yp\alpha + \frac{x+yq}{2}} = \alpha,
\end{align*}
and are invertible because
\begin{align*}
\begin{vmatrix}
\frac{x - yq}{2} & -yr\\
yp & \frac{x + yq}{2}
\end{vmatrix}
&= \frac{x^2 - y^2q^2}{4} + y^2pr\\
&= \frac{1}{4}(x^2 - Dy^2)\\
&= \pm 1,
\end{align*}
where the sign is positive if and only if $(x,y)$ satisfies $x^2 - Dy^2 = +4$. To see that $\Gamma$ is a group homomorphism, we check:
\begin{align*}
\Gamma\left(\frac{x+y\sqrt{D}}{2} \cdot \frac{x'+y'\sqrt{D}}{2}\right) &= \Gamma\left(\frac{\frac{xx' + Dyy'}{2} + \frac{xy' + x'y}{2}\sqrt{D}}{2}\right)\\
&=
\begin{bmatrix}
\frac{xx' + yy'(q^2 - 4pr) - xy'q -x'yq}{4} & \frac{-xy'r - x'yr + yy'qr - yy'qr}{2}\\
\\
\frac{x'yp + xy'p + yy'pq - yy'pq}{2} & \frac{xx' + yy'(q^2-4pr) + xy'q + x'yq}{4}
\end{bmatrix}\\
&=
\begin{bmatrix}
\frac{x-yq}{2} & -yr\\
yp & \frac{x+yq}{2}
\end{bmatrix}
\begin{bmatrix}
\frac{x'-y'q}{2} & -y'r\\
y'p & \frac{x'+y'q}{2}
\end{bmatrix}\\
&= \Gamma\left(\frac{x+y\sqrt{D}}{2}\right) \Gamma\left(\frac{x'+y'\sqrt{D}}{2}\right).
\end{align*}

Combining the map $\Gamma$ with the isomorphism of Proposition \ref{isomstab} we have now shown how to construct, for each solution to (\ref{eqn:pell}), a self-similarity of $\PP_{\alpha}$. To prove our Theorem, we calculate the kernel and image of the homomorphism $\Gamma$. First, we show that $\Gamma$ is surjective, which implies that every self-similarity of $\PP_{\alpha}$ arises from a solution to (\ref{eqn:pell}) in the manner described above. So suppose we are given a matrix
$A = \begin{bmatrix}
a & b\\
c & d
\end{bmatrix}$
that stabilizes $\alpha$. In particular, it follows that
\[ c\alpha^2 + (d-a)\alpha - b = 0 \]
This polynomial is therefore an integer multiple of the primitive polynomial $px^2 + qx + r$ with root $\alpha$. That is, there exists $m \in \mathbb{Z}$ such that
\begin{align*}
c &= mp  \\
d-a &= mq  \\
-b &= mr.
\end{align*}
Now set
\begin{align*}
x &= a+d \\
y &= m.
\end{align*}
We clearly have $x,y \in \zz$ and
\begin{align*}
x^2 - Dy^2 &= (a+d)^2 - (q^2 - 4pr)m^2 \\
&= (a+d)^2 - (a-d)^2 - 4bc \\
&= 4(ad - bc) = \pm 4.
\end{align*}
so $\frac{x + y\sqrt{D}}{2} \in \GG$. (Moreover, this is in $\GG_+$ if and only if $A \in \PSL_2(\zz)$.) It is easy to check that $\Gamma(\frac{x + y\sqrt{D}}{2}) = A$ as required.

Finally, $\frac{x + y\sqrt{D}}{2}$ is in the kernel of $\Gamma$ if and only if
\[ x - yq = x + yq = \pm 2 \ , \ yr = yp = 0. \]
Since $p$ cannot be zero ($\alpha \not \in \qq$), it follows that $y = 0$, and therefore $x = \pm 2$. In other words,
\begin{equation*}
\ker(\Gamma) = \{\pm 1\}.
\end{equation*}

Putting together our various isomorphisms and using \eqref{eqn:pellgroup}, we now have
\[ \Symm(\PP_{\alpha}) \cong \GG / \{\pm 1\} \cong \mathbb{Z}. \]
We have also seen that the orientation-preserving self-similarities correspond under this isomorphism to the subgroup $\GG_+ / \{\pm 1\}$. There are two possibilities here. One is that the generator for $\GG$ is in $\GG_+$. In this case the groups are equal and all the self-similarities of $\PP_{\alpha}$ are orientation-preserving. This happens when there are no integer solutions to the equation
\[ x^2 - Dy^2 = -4. \]
The other possibility is that the generator $z = \frac{x_0 + y_0\sqrt{D}}{2}$ for $\GG$ is not in $\GG_+$. But then, however, $z^2$ is in $\GG_+$ and so $\GG_+$ is an index 2 subgroup of $\GG$. In this case, $\Symm_+(\PP_{\alpha})$ is an index 2 subgroup in $\Symm(\PP_{\alpha})$ as claimed.
\end{proof}

\begin{cor}
The half-plane packing $\PP_{\alpha}$ is self-similar if and only if $\alpha$ is rational or quadratic over $\mathbb{Q}$. In the quadratic case, $\PP_{\alpha}$ is self-similar via an orientation-reversing self-similarity if and only if the equation $x^2 - Dy^2 = -4$ has an integral solution $(x,y)$, where $D$ is the discriminant of the primitive integral polynomial with root $\alpha$.
\end{cor}

\section{Half-plane Packings and Continued Fractions}\label{sec:contfrac}

In this section, our main goal is to describe how the continued fraction of a positive real number $\alpha$ manifests itself geometrically in the half-plane packing $\PP_{\alpha}$. Recall that the standard continued fraction expansion of a positive real number $\alpha$ is a representation of the form
\[ \alpha = a_0 + \cfrac{1}{a_1 + \cfrac{1}{a_2 + \cfrac{1}{\ddots}}} \]
for some integers $a_k$, with $a_0 \ge 0$ and $a_k > 0$ for all $k > 0$. It is also standard (and more practical typographically) to express this expansion simply as $\alpha = [a_0,a_1,a_2,\ldots]$.

To begin, we recall the algorithm for computing the continued fraction expansion of a positive real number $\alpha$. The continued fraction expansion is computed by successive iterations of the following algorithm, which we refer to as the \textbf{continued fraction algorithm}. The input for the algorithm is the number $\alpha_0 = \alpha$. Each step of the algorithm takes $\alpha_n$ and determines $\alpha_{n+1}$.
\begin{enumerate} \renewcommand{\labelenumi}{(\Alph{enumi})}
\item If $\alpha_n \geq 1$, let $\alpha_{n+1} = \alpha_n - 1$.
\item If $0 < \alpha_n < 1$, let $\alpha_{n+1} = \frac{1}{\alpha_n}$.
\item If $\alpha_n = 0$, halt.
\end{enumerate}
Recording the sequence of steps obtained when applying this algorithm to a positive real number $\alpha$ we get something like
\[ ABAABAAC. \]
The positive integer $a_k$ from the continued fraction expansion corresponds precisely to the length of the $(k+1)$th string of consecutive $A$'s. For example, the above sequence represents the application of the continued fraction algorithm to $\alpha_0 = \frac{7}{5}$. The resulting continued fraction expansion is
\[ [1,2,2] = 1 + \cfrac{1}{2 + \cfrac{1}{2}} = \frac{7}{5}. \]
The sequence $(\alpha_n)$ in this case is:
\begin{equation}
\frac{7}{5} \xrightarrow{\mbox{(A)}} \frac{2}{5} \xrightarrow{\mbox{(B)}} \frac{5}{2}\xrightarrow{\mbox{(A)}} \frac{3}{2} \xrightarrow{\mbox{(A)}} \frac{1}{2} \xrightarrow{\mbox{(B)}} 2 \xrightarrow{\mbox{(A)}} 1 \xrightarrow{\mbox{(A)}} 0.
\label{cfe7/5}
\end{equation}
Notice that the continued fraction expansion for $\alpha_n$ is the same as that for $\alpha$, but with an `initial segment' removed. For example, if $\alpha_0 = [2,3,4,5,6,7]$, then we have
\[ \begin{split}
  \alpha_1 = [1,3,4,5,6,7] \\
  \alpha_2 = [0,3,4,5,6,7] \\
  \alpha_3 = [3,4,5,6,7] \\
  \alpha_4 = [2,4,5,6,7] \\
  \alpha_5 = [1,4,5,6,7] \\
  \alpha_6 = [0,4,5,6,7] \\
  \alpha_7 = [4,5,6,7]
\end{split} \]
and so on.

Turning now back to Apollonian circle packings, we define a \textbf{circle replacement algorithm}. The input of this algorithm is an ordered pair $(X_0,Y_0)$, where $X_0$ and $Y_0$ are tangent circles in a half-plane packing $\PP$ that are also tangent to a chosen line $L \in \PP$. We also require that $Y_0$ is not itself a line. At the $(n+1)$th step of the algorithm, we replace the pair $(X_n,Y_n)$ with a new pair of circles $(X_{n+1},Y_{n+1})$:
\begin{enumerate} \renewcommand{\labelenumi}{(\Alph{enumi})}
\item If $\curv(X_n) \geq \curv(Y_n)$, take $Y_{n+1} = Y_n$ and take $X_{n+1}$ to be the circle that fills the unbounded interstice for $X_n$ and $Y_n$ (in the sense of Definition \ref{defn:pairinterstice}). Note that Corollary \ref{descartesline} implies that
    \[ \sqrt{\curv(X_{n+1})} = \sqrt{\curv(X_n)} - \sqrt{\curv(Y_n)}. \]
\item If $0 < \curv(X_n) < \curv(Y_n)$, take $X_{n+1} = Y_n$ and $Y_{n+1} = X_n$.
\item If $\curv(X_n) = 0$, halt.
\end{enumerate}

As with the continued fraction algorithm, we are interested in the sequence of steps involved when the algorithm is performed to a given starting pair of circles. (For example, we might obtain the sequence $AABABAAAABAAC$.) Our main observation is then the following.

\begin{lem} \label{lem:cfe}
Let $\alpha$ be a positive real number, and let $(X_0,Y_0)$ be the two circles used to construct the half-plane packing $\PP_{\alpha}$: $X_0$ and $Y_0$ are tangent to each other and to the $x$-axis $L$, and $\curv(X_0) = \alpha^2$, $\curv(Y_0) = 1$. Then the sequence of steps (A, B, or C) performed in applying the continued fraction algorithm to $\alpha$ is the same as the sequence of steps performed in applying the circle replacement algorithm to $(X_0,Y_0)$. Moreover, we have
\[ \alpha_n = \frac{\sqrt{\curv(X_n)}}{\sqrt{\curv(Y_n)}} \]
for all $n \geq 0$.
\end{lem}
\begin{proof}
The proof is by induction on $n$. For $n = 0$, this is the claim
\[ \alpha = \frac{\sqrt{\curv(X_0)}}{\sqrt{\curv(Y_0)}} \]
which is true by the choice of $X_0$ and $Y_0$.

Suppose that the claim holds for $\alpha_n$ and $(X_n,Y_n)$. Then $\alpha_n \geq 1$ if and only if $\curv(X_n) \geq \curv(Y_n)$ and $\alpha_n = 0$ if and only if $\curv(X_n) = 0$. This tells us that the next step (A, B, or C) will be the same for both algorithms. So it remains only to verify that the formula still holds for $\alpha_{n+1}$ and $(X_{n+1},Y_{n+1})$.

Suppose that $\alpha_n \geq 1$. Then we have $\alpha_{n+1} = \alpha_n - 1$, so it is sufficient to show that
\[ \frac{\sqrt{\curv(X_{n+1})}}{\sqrt{\curv(Y_{n+1})}} = \frac{\sqrt{\curv(X_n)}}{\sqrt{\curv(Y_n)}} - 1. \]
We have $Y_{n+1} = Y_n$, so it is enough to show that
\[ \sqrt{\curv(X_{n+1})} = \sqrt{\curv(X_n)} - \sqrt{\curv(Y_n)} \]
which follows from Corollary \ref{descartesline} as mentioned above.

Finally, suppose that $0 < \alpha_n < 1$. Then
\[ \alpha_{n+1} = \frac{1}{\alpha_n} = \frac{\sqrt{\curv(Y_n)}}{\sqrt{\curv(X_n)}} = \frac{\sqrt{\curv(X_{n+1})}}{\sqrt{\curv(Y_{n+1})}}. \]
\end{proof}

Figure \ref{fig:cfe7/5} shows the circle replacement algorithm applied to the packing $\PP_{\alpha}$ for $\alpha = \frac{7}{5}$. The circles are labeled by the square roots of their curvatures. Compare this to \eqref{cfe7/5} as an illustration of Lemma \ref{lem:cfe}.

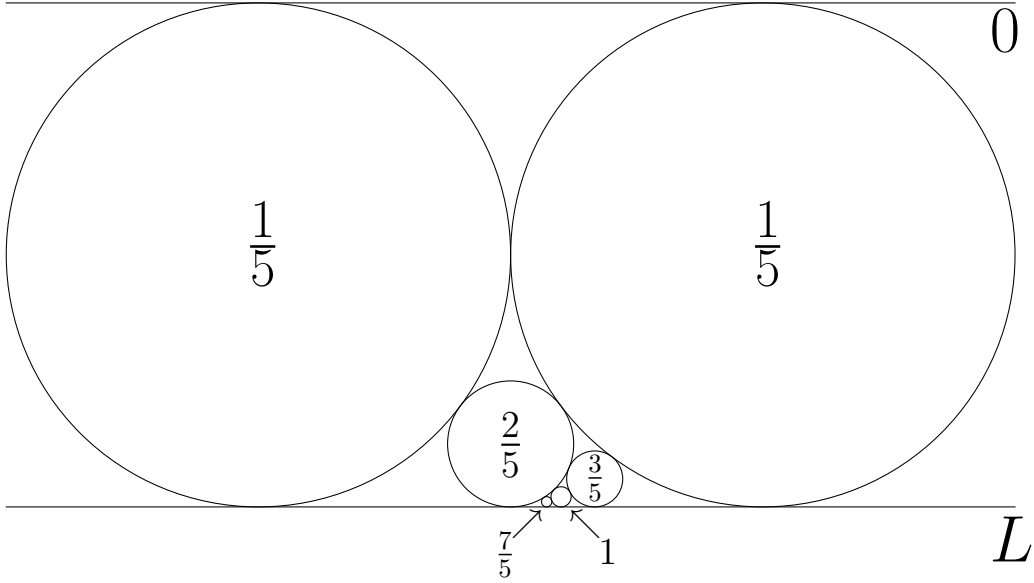
\begin{figure}[h]
	\begin{overpic}[scale=.99]{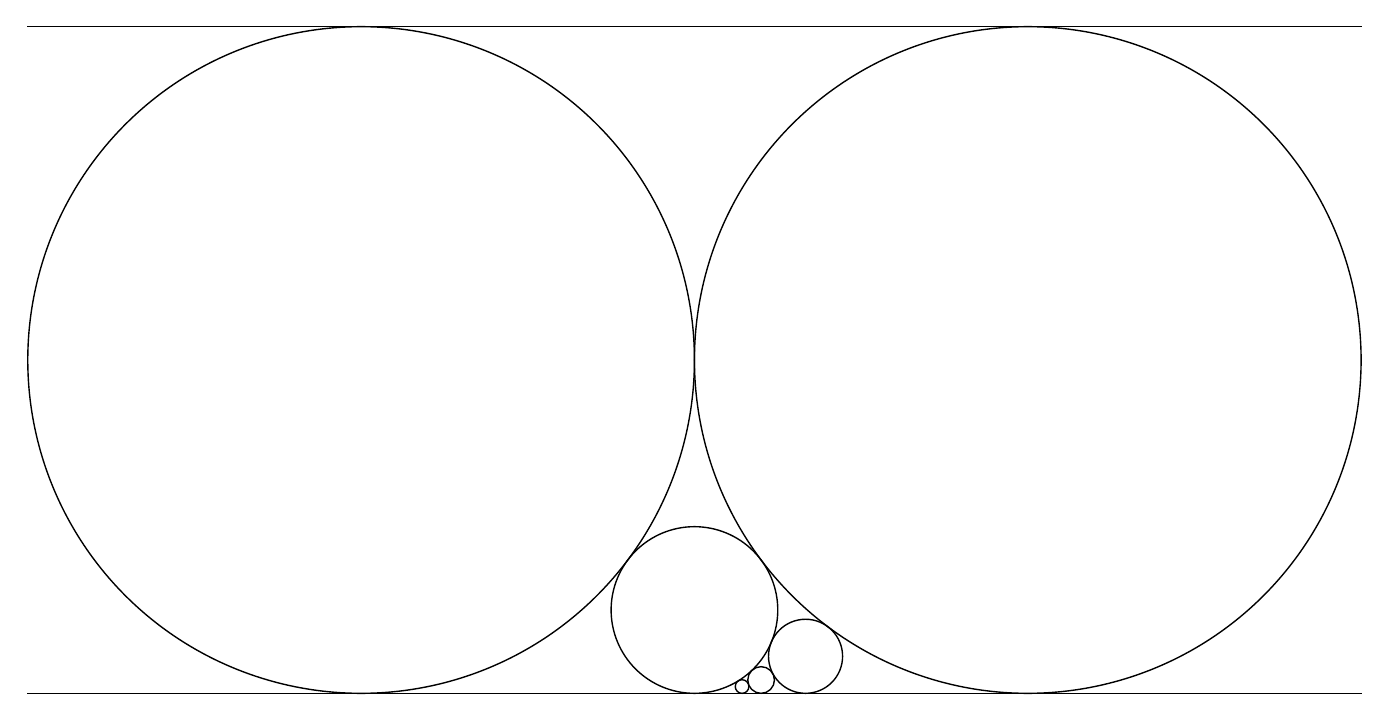}
	\put(199,-2){\large $\nearrow$}
	\put(221,-2){\large $\nwarrow$}
	\put(193,-14){\Large $\frac{7}{5}$}
	\put(232,-14){\Large $1$}
	\put(227,15){\Large $\frac{3}{5}$}
	\put(194,26){\huge $\frac{2}{5}$}
	\put(99,100){\Huge $\frac{1}{5}$}
	\put(290,100){\Huge $\frac{1}{5}$}
	\put(380,180){\Huge $0$}
	\put(380,-13){\Huge $L$}
	\end{overpic}
	
\vspace{5mm}	
	
	\caption{The circle replacement algorithm applied to the packing $\PP_{\frac{7}{5}}$. The circles are labeled by the square roots of their curvatures.}
  \label{fig:cfe7/5}
\end{figure}

Recall that the continued fraction expansion of a real number $\alpha$ determines a sequence of rational numbers $\frac{p_n}{q_n}$ that converge to $\alpha$. These are the \emph{convergents} of $\alpha$ and are given by truncating the continued fraction expansion of $\alpha$. Thus if
\[ \alpha = [a_0,a_1,a_2,\dots] \]
then set
\[ \frac{p_n}{q_n} := [a_0,a_1,\dots,a_n], \]
where $p_n$ and $q_n$ are nonnegative coprime integers. They satisfy the recurrence equations
\begin{equation}
p_n = p_{n-2} + a_np_{n-1}, \quad q_n = q_{n-2} + a_nq_{n-1}.
\label{eqn:recurrence}
\end{equation}

We now observe that the convergents of $\alpha$ appear in the labels (in the sense of \S\ref{sec:acp}) of the circles in the circle replacement algorithm applied to the packing $\PP_{\alpha}$.

\begin{lem} \label{lem:convergents}
Let $\alpha$ be a positive real number. Let $(X_0,Y_0) = (C_{(1,0)},C_{(0,1)})$ be the generating circles for the packing $\PP_{\alpha}$. The sequence of \emph{distinct} circles in the sequence $(Y_j)$ defined by the circle replacement algorithm is
\[ C_{(0,1)}, C_{(q_0,-p_0)}, C_{(-q_1,p_1)}, C_{(q_2,-p_2)}, C_{(-q_3,p_3)}, \dots. \]
In particular, if $\alpha \not\in \mathbb{Q}$, then
\[ \lim_{j \to \infty} \curv(Y_j) = 0. \]
\end{lem}
\begin{proof}
We have $Y_0 = C_{(0,1)}$ and $X_0 = C_{(1,0)}$. The first new $Y_j$ will appear after the first application of step (B) of the algorithm, that is, after $a_0+1$ steps. At this point we have
\[ Y_{a_0+1} = X_{a_0} = C_{(1,-a_0)} = C_{(q_0,-p_0)} \]
and
\[ X_{a_0+1} = Y_{a_0} = C_{(0,1)}. \]
Now suppose, inductively, that immediately after the $n$th application of step (B) we have
\[ Y_N = C_{(q_{n-1},-p_{n-1})}, X_N = C_{(-q_{n-2},p_{n-2})}. \]
Running the algorithm until after the next application of (B), that is $a_n+1$ times, we have
\[ Y_{N+a_n+1} = X_{N + a_n} = C_{(-q_{n-2}-a_nq_{n-1},p_{n-2}+a_np_{n-1})} = C_{(-q_n,p_n)} \]
and
\[ X_{N+a_n+1} = Y_{N+a_n} = Y_N = C_{(q_{n-1},-p_{n-1})}. \]
The first claim now follows by induction on $n$.

By Lemma \ref{curv}, we have
\[ \curv(C_{(\pm q_n,\mp p_n)}) = (q_n \alpha - p_n)^2 = q_n^2 \left( \alpha - \frac{p_n}{q_n} \right)^2. \]
A basic fact about the convergents for continued fractions \cite[Theorem 171]{hardy/wright:2008} is that
\[ \left|\alpha - \frac{p_n}{q_n}\right| < \frac{1}{q_n^2} \]
so
\[ \curv(C_{(\pm q_n,\mp p_n)}) < \frac{1}{q_n^2}. \]
It follows from \eqref{eqn:recurrence} that $q_n \to \infty$ as $n \to \infty$, so the curvatures of the $Y_j$ tend to zero.
\end{proof}

Our goal is now to use this relationship between the circle replacement and continued fraction algorithms to give new criteria for two packings to be similar, and a new way to understand the self-similarities of a given packing, both in terms of continued fraction expansions. To do this we have to know that the circle replacement algorithm involves `enough' of the circles in the packing to be able to detect any similarity between two packings. The following lemma is key to this.

\begin{lem}\label{lem:replacementcircles}
Let $\PP$ be a half-plane packing (but not a strip packing). Fix an ordered pair of circles $(X_0,Y_0)$ as in the definition of the circle replacement algorithm. Let $X$ and $Y$ be any pair of tangent circles in $\PP_L$ such that $X_0$ and $Y_0$ are contained in the bounded interstice formed by $X$ and $Y$. Then one of the pairs $(X,Y)$ and $(Y,X)$ appears as $(X_n,Y_n)$ in the application of the circle replacement algorithm to $(X_0,Y_0)$.
\end{lem}
\begin{proof}
The second part of Lemma \ref{lem:convergents} implies that the radii of the circles $X_n$ and $Y_n$ increase without bound as $n$ tends to infinity. Therefore there is some smallest integer $N$ such that $X_N$ does not lie in the bounded interstice formed by $X$ and $Y$. Since $X_N$ is tangent to $X_{N-1}$, the remark following Definition \ref{defn:pairinterstice} implies that $X_N$ also cannot lie in the unbounded interstice for $X$ and $Y$. Hence $X_N$ is equal to one of $X,Y$, say $Y$ without loss of generality.

Now $Y_N$ is either equal to $X_n$ for some $n < N$, or is equal to $Y_0$. Either way, $Y_N$ is in the bounded interstice formed by $X$ and $Y$ and so, in particular, is smaller than $X_N$. This tells us that the next step in the circle replacement algorithm is (B); i.e., $X_{N+1} = Y_N$ and $Y_{N+1} = X_N = Y$. Since $Y_{N+1}$ is larger than $X_{N+1}$, we next repeat step (A) until $X_{N+K}$ is larger than $Y_{N+K} = Y_{N+1}$ for some $K \ge 0$. Then $X_{N+K}$ is not in the bounded interstice formed by $X$ and $Y$, but $X_{N+1}$ is. This means we can find a smallest $M$ with $N+1 < M \leq N+K$ such that $X_M$ is not in the bounded interstice formed by $X$ and $Y$. Since it is tangent to $X_{M-1}$, this circle $X_M$ also cannot be in the unbounded interstice, so must be one of $X$ and $Y$. But it is not $Y$ since $Y_M$ is. Therefore we have $X_M = X$ and $Y_M = Y$ which completes the proof.
\end{proof}

We can now relate properties of the continued fraction expansion of a positive real number $\alpha$ to geometric properties of the half-plane circle packing $\PP_{\alpha}$.

\subsection{Strip packings}
We already saw in Corollary \ref{strip} that $\PP_{\alpha}$ is the strip packing if and only if $\alpha \in \mathbb{Q}$. This is now reflected in the fact that the continued fraction expansion for $\alpha$ halts if and only if $\alpha \in \mathbb{Q}$. We can see from Lemma \ref{lem:cfe} that the continued fraction expansion of $\alpha$ halts exactly when the corresponding circle replacement algorithm produces a circle of curvature $0$, that is, a straight line. This is illustrated in the example of $\alpha = \frac{7}{5}$ displayed above.

\subsection{Similar packings}
We can determine whether the packings $\PP_{\alpha}$ and $\PP_{\beta}$ are similar by examining the tails of the continued fractions of $\alpha$ and $\beta$.
\begin{defn}
Let us say that $\alpha$ and $\beta$ have \textbf{eventually equal} continued fraction expansions if there is some $k,N \in \mathbb{Z}$ such that $a_n = b_{n+k}$ for all $n \geq N$ (where $[a_i]$ is the continued fraction expansion of $\alpha$ and $[b_j]$ is the continued fraction expansion of $\beta$). For example, $[1,2,3,4,4,4,4,\dots]$ and $[2,7,4,4,4,4,\dots]$ are eventually equal. For the purposes of Theorem \ref{thm:eventuallyequal}, we say that two finite continued fraction expansions are eventually equal since both expansions terminate.
\end{defn}

It is easy to see that $\alpha$ and $\beta$ have eventually equal continued fraction expansions if and only if there exist $m,n \geq 0$ such that $\alpha_n = \beta_m$ (where these are the sequences obtained by applying the continued fraction algorithm to $\alpha$ and $\beta$). This observation yields the following classification of half-plane packings up to similarity.

\begin{thm}\label{thm:eventuallyequal}
For positive real numbers $\alpha,\beta$, the circle packings $\PP_{\alpha}$ and $\PP_{\beta}$ are similar if and only if $\alpha$ and $\beta$ have eventually equal continued fraction expansions.
\end{thm}
\begin{proof}
A quick proof of this result follows by identifying each of the conditions in the statement with the condition that there exist $p,q,r,s \in \mathbb{Z}$ with $ps - qr = \pm 1$ and $\frac{p\alpha + q}{r\alpha + s} = \beta$. For the circle packings, this is Lemma \ref{linfracequiv++}; for the continued fractions, it is \cite[Theorem 175]{hardy/wright:2008}. However, we give a more interesting proof arising from the direct comparison between the continued fraction and circle replacement algorithms.

First, note that $\alpha$ and $\beta$ have finite continued fraction expansions if and only if $\alpha$ and $\beta$ are rational, which is equivalent by Corollary \ref{strip} to $\PP_{\alpha}$ and $\PP_{\beta}$ both being strip packings, which are similar. We may therefore assume that $\alpha$ and $\beta$ are irrational.

Suppose $\alpha$ and $\beta$ have eventually equal continued fraction expansions. Then $\alpha_n = \beta_m$ for some $m,n$. This means that the ratio of the curvatures of a pair of tangent circles in $\PP_{\alpha}$, both tangent to $L$, is equal to the ratio of the curvatures of a pair of tangent circles in $\PP_{\beta}$, both tangent to $L$. It follows by Lemma \ref{similar} that there is a similarity between $\PP_{\alpha}$ and $\PP_{\beta}$.

To prove the converse, suppose $\PP_{\alpha}$ and $\PP_{\beta}$ are similar. Then there is a pair of circles $(X_0',Y_0')$ in $\PP_{\alpha}$ whose ratio of curvatures is equal to $\beta^2$, in addition to the original pair of circles $(X_0,Y_0)$ in $\PP_{\alpha}$ whose ratio of curvatures is $\alpha^2$. The key step is the following claim: if we apply the circle replacement algorithm to each of these pairs of circles, they will eventually coincide; that is, there is some pair of circles $(X,Y)$ in $\PP_{\alpha}$ that appears both as $(X_M,Y_M)$ and $(X'_N,Y'_N)$ for some $M,N \in \mathbb{N}$. Note that the circle replacement algorithm only sees \emph{ratios} of curvatures and not the curvatures themselves, so the circle replacement algorithm will generate the same numerical data for $(X',Y')$ as it would for the corresponding pair of circles in the packing $\PP_{\beta}$. From this claim, it follows that $\alpha$ and $\beta$ have eventually equal continued fraction expansions since
$$\alpha_M = \frac{\sqrt{\curv(X_M)}}{\sqrt{\curv(Y_M)}} = \frac{\sqrt{\curv(X)}}{\sqrt{\curv(Y)}} = \frac{\sqrt{\curv(X_N')}}{\sqrt{\curv(Y_N')}} = \beta_N.$$

Now let us prove that the circle replacement algorithms corresponding to $\alpha$ and $\beta$ eventually coincide in the sense described in the previous paragraph. Let $(X_n,Y_n)$ be the pairs obtained from applying the circle replacement algorithm to $(X_0,Y_0)$, as defined in the previous paragraph. We suppose without loss of generality that $X_0 \prec Y_0$ and $X'_0 \prec Y'_0$. If one of the pairs contained the other in its bounded interstice, say $X_0 \preceq X'_0 \prec Y'_0 \preceq Y_0$, then by Lemma \ref{lem:replacementcircles} there would be $M$ such that $\{X_M,Y_M\} = \{X'_0,Y'_0\}$. Suppose instead that
\[ X_0 \prec Y_0 \preceq X'_0 \prec Y'_0. \]
Suppose that there is no pair $(X_n,Y_n)$ that contains $(X'_0,Y'_0)$ in its bounded interstice. Then, for each $n$, one of the circles $X_n$ and $Y_n$, has its point of tangency with the $x$-axis between the corresponding tangency points of $Y_0$ and $X'_0$. By the second part of Lemma \ref{lem:convergents}, this means that there are arbitrarily large circles, all disjoint, with tangency points in this fixed interval. A little geometry shows that if two disjoint circles of radii $R$ and $R'$ are tangent to the $x$-axis, then their points of tangency are at least $2\sqrt{RR'}$ apart. This gives us a contradiction and so we deduce that there is $M$ such that $X_M$ and $Y_M$ contain both $X'_0$ and $Y'_0$ in their bounded interstice. But then by Lemma \ref{lem:replacementcircles}, the pair $\{X_M,Y_M\}$ is equal to $\{X'_N,Y'_N\}$ for some $N$. To complete the proof that the algorithms eventually coincide, we need to show that we can choose $M$ and $N$ such that $X_M = X'_N$ and $Y_M = Y'_N$.

Suppose instead that $X_M = Y_N'$ and $Y_M = X_N'$. We may assume that $X_M$ is smaller than $Y_M$ (otherwise, apply one more replacement to $(X_M,Y_M)$ to replace $X_M$ with $Y_{M+1}$ and $Y_M$ with $X_{M+1}$). It follows that $Y_N'$ is smaller than $X_N'$, so that the next step of the algorithm, step (B), will set $X_{N+1}' = Y_N'$ and $Y_{N+1}' = X_N'$, and we therefore have $X_M = X_{N+1}'$ and $Y_M = Y_{N+1}'$.
\end{proof}

\subsection{Self-similar packings} We also have already seen (Theorem \ref{thm}) that $\PP_{\alpha}$ is self-similar (but not the strip packing) if and only if $\alpha$ is quadratic over $\qq$. It is a well-known fact that an irrational number $\alpha$ is quadratic over $\qq$ if and only if its continued fraction expansion is infinite and periodic; i.e., if and only if
\[ \alpha = [a_0,\ldots,a_{n-1},c_0,\ldots,c_{m-1},c_0,\ldots,c_{m-1},\ldots] = [a_0,\ldots,a_{n-1},\overline{c_0,\ldots,c_{m-1}}]. \]

This fact, along with Theorem \ref{thm:eventuallyequal}, yields the following classification of self-similar half-plane packings (which are not strip packings):

\begin{thm}\label{thm:classification}
The similarity classes of self-similar half-plane (non-strip) circle packings correspond bijectively to finite non-repeating sequences of positive integers, up to cyclic permutation. (A sequence is non-repeating if it is not equal to the concatenation of multiple copies of the same smaller sequence.)
\end{thm}

\begin{proof}
$\PP_{\alpha}$ is self-similar if and only if the continued fraction expansion for $\alpha$ is periodic, as we mentioned above. We identify the similarity class of $\PP_{\alpha}$ with the minimal periodic part of this expansion. For example, if $\alpha = \sqrt{2} = [1,2,2,2,\dots]$ then we identify $[\PP_{\alpha}]$ with the one term sequence $(2)$. If $\alpha = \sqrt{3} = [1,1,2,1,2,1,2,\dots]$, we identify $[\PP_{\alpha}]$ with $(1,2)$, or equivalently, $(2,1)$. Conversely, the finite non-repeating sequence $(a_0,\ldots,a_{n-1})$ represents the quadratic number $\alpha = [\overline{a_0,\ldots,a_{n-1}}]$, so that every such sequence represents a similarity class of self-similar half-plane packings (namely, the class including $\PP_{\alpha}$). That each sequence represents exactly one similarity class follows from Theorem \ref{thm:eventuallyequal}, since two periodic continued fractions expansions are eventually equal if and only if they have the same periodic part up to a cyclic permutation.
\end{proof}

\begin{exs}
Based on the classification in Theorem \ref{thm:classification} we can give examples of the simplest self-similar half-plane packings. In some sense, the simplest such packing is given by $\alpha = [1,1,\dots] = \frac{1+\sqrt{5}}{2}$. From the perspective of Theorem \ref{thm:classification}, this is represented by the singleton sequence (1). The corresponding circle packing $\PP_{\alpha}$ has a self-similarity constructed from a single circle replacement. This is displayed in Figure \ref{fig:selfsim1}.

The next simplest example is $\alpha = [2,2,\dots] = 1+\sqrt{2}$, represented by the singleton sequence (2). The corresponding circle packing $\PP_{\alpha}$ has a self-similarity obtained by doing two circle replacements. This appears in Figure \ref{fig:selfsim2}.

There are two different self-similar packings for which a self-similarity involves three circle replacements. Corresponding to the sequence (3), we have $\alpha = [3,3,\dots] = \frac{3+\sqrt{13}}{2}$. The packing $\PP_{\alpha}$ is shown in Figure \ref{fig:selfsim3}. Corresponding to the sequence (1,2) we have $\alpha = [1,2,1,2,\dots] = \frac{1+\sqrt{3}}{2}$. This packing is shown in Figure \ref{fig:selfsim12}.
\end{exs}

\begin{figure}
\begin{overpic}[scale = .75]{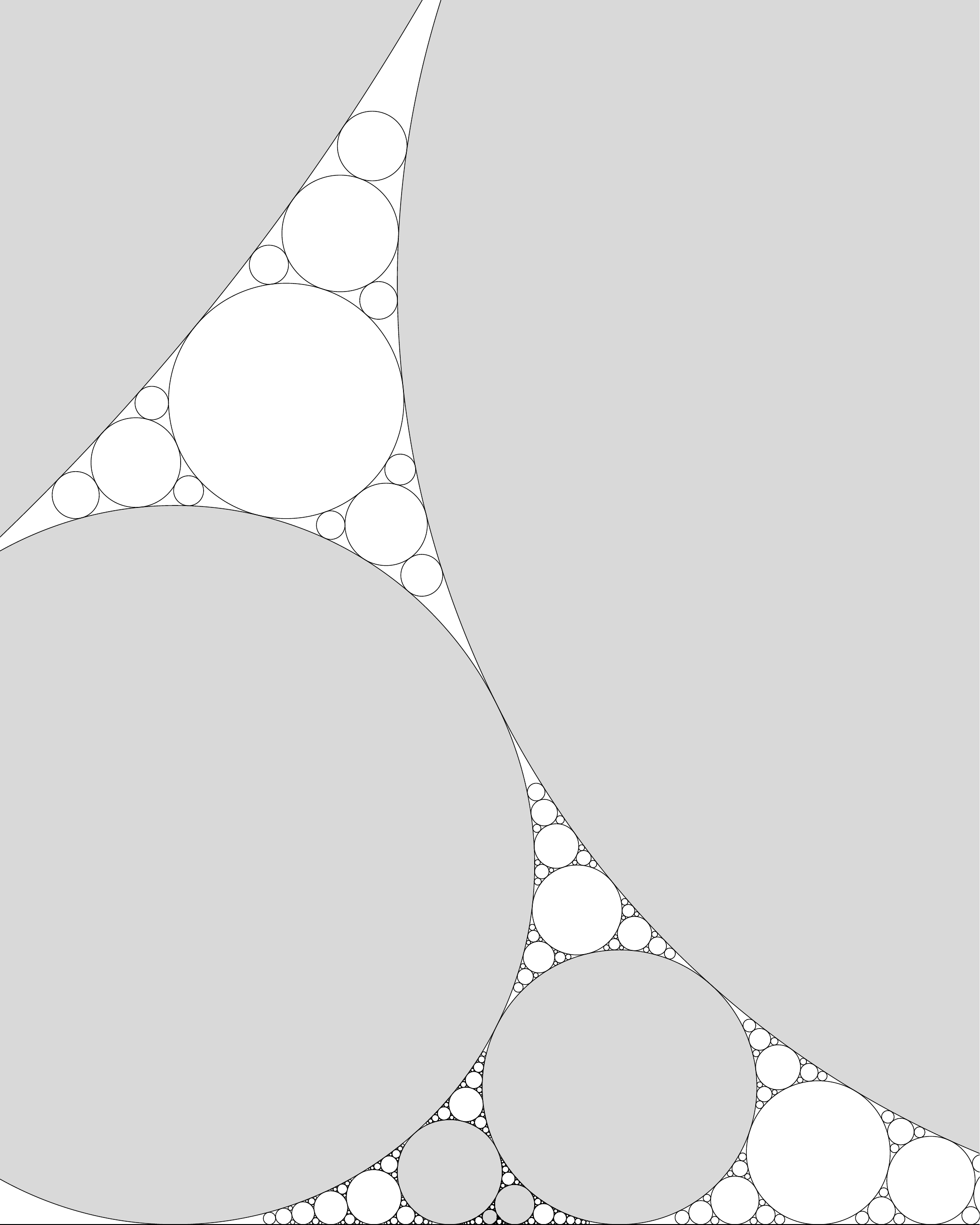}
	\put(223,-10){\large $\uparrow$}
	\put(238,-10){\large $\nwarrow$}
	\put(218,-24){\Large $X_0$}
	\put(249,-24){\Large $Y_0$}
\end{overpic}
	
\vspace{10mm}	
	
	\caption{Eight generations of the packing $\PP_{\alpha}$, where $\alpha = [\overline{1}] = \frac{1 + \sqrt{5}}{2}$. The circles corresponding to the circle replacement algorithm are shaded gray.}
	\label{fig:selfsim1}
\end{figure}

\begin{figure}
\begin{overpic}[scale = .75]{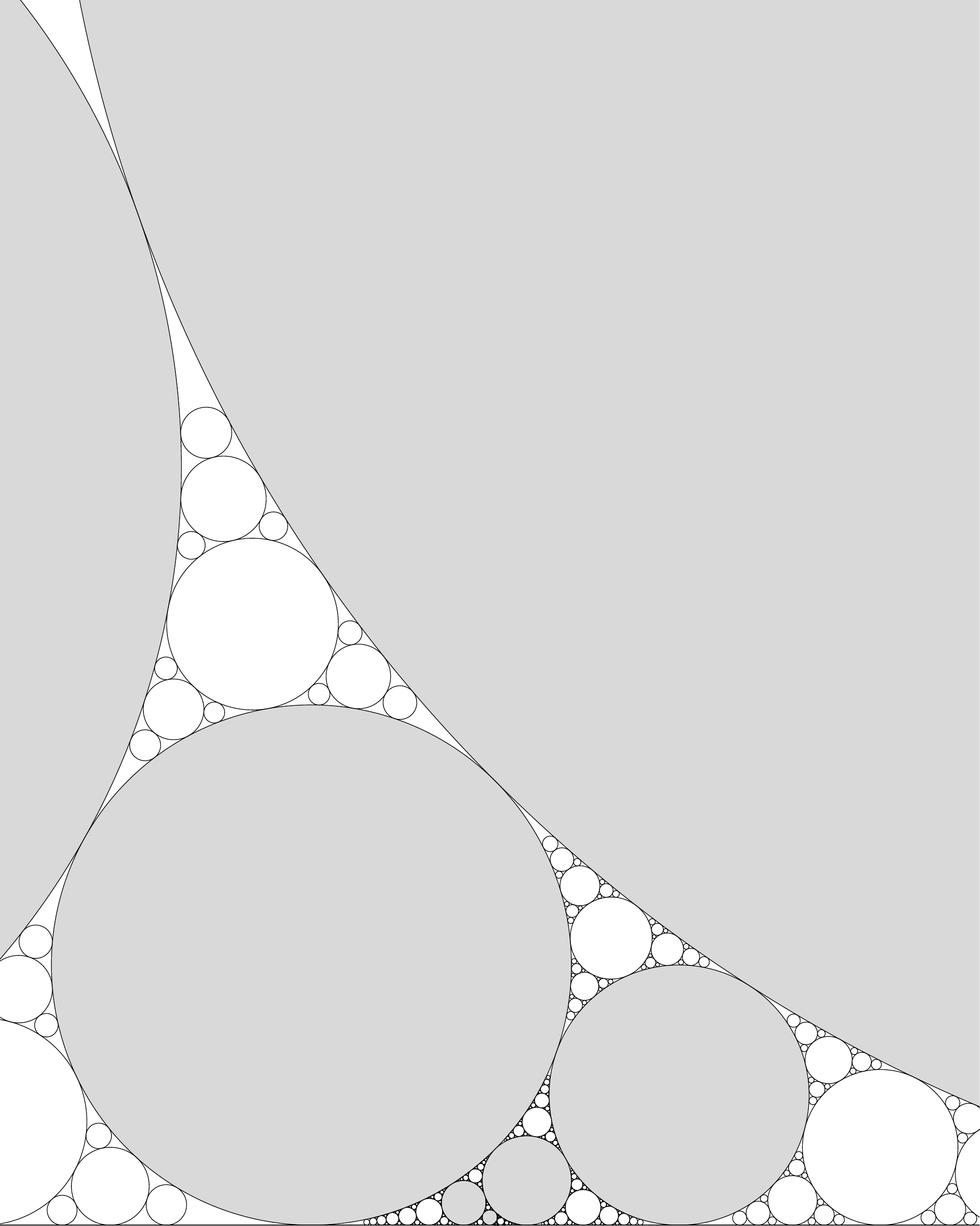}
	\put(223,-10){\large $\uparrow$}
	\put(242,-10){\large $\nwarrow$}
	\put(218,-24){\Large $X_0$}
	\put(253,-24){\Large $Y_0$}
\end{overpic}
	
\vspace{10mm}	
	
	\caption{Eight generations of the packing $\PP_{\alpha}$, where $\alpha = [\overline{2}] = 1 + \sqrt{2}$. The circles corresponding to the circle replacement algorithm are shaded gray.}
	\label{fig:selfsim2}
\end{figure}

\begin{figure}
\begin{overpic}[scale = .75]{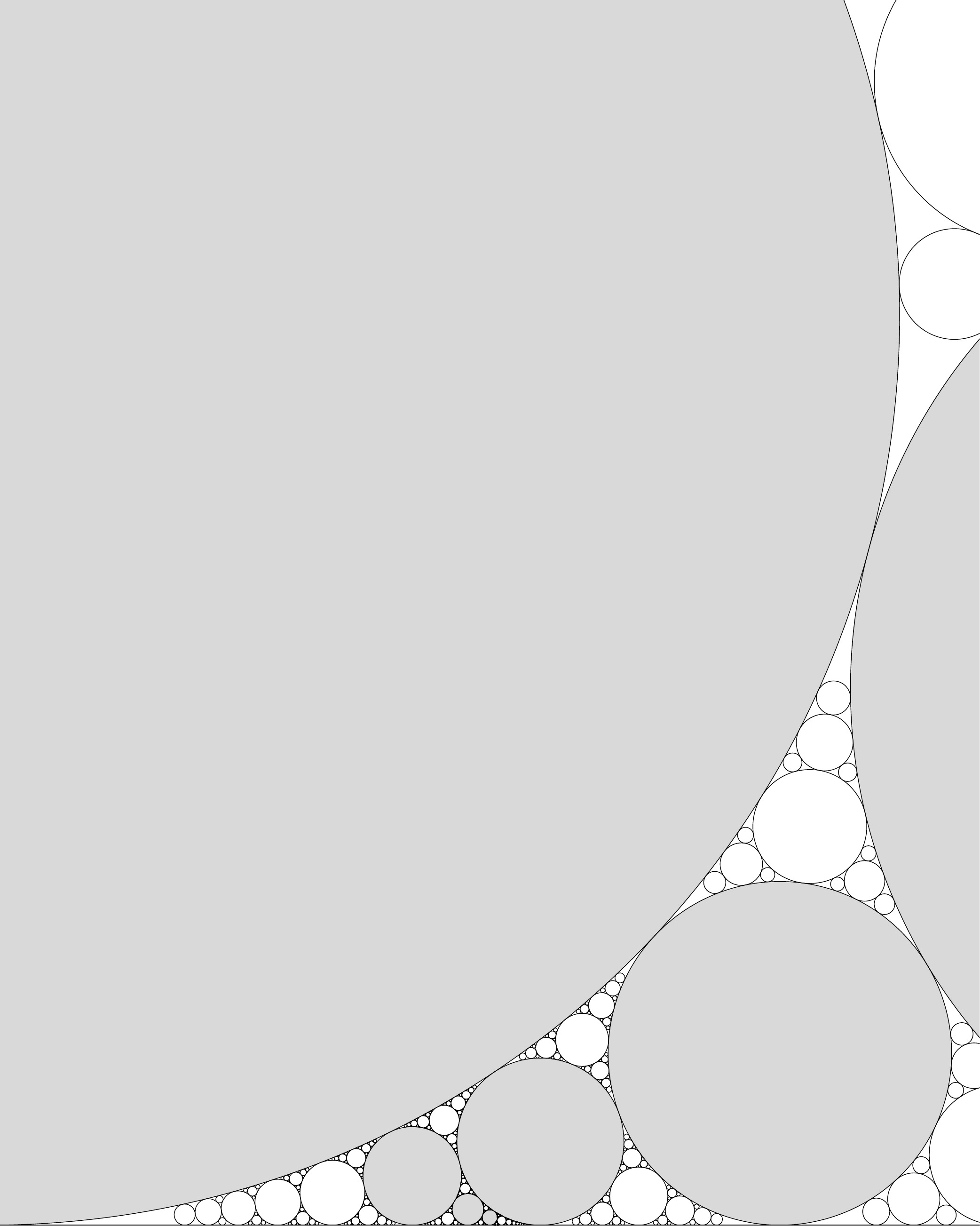}
	\put(223,-10){\large $\uparrow$}
	\put(245,-10){\large $\uparrow$}
	\put(218,-24){\Large $X_0$}
	\put(242,-24){\Large $Y_0$}
\end{overpic}
	
\vspace{10mm}	
	
	\caption{Eight generations of the packing $\PP_{\alpha}$, where $\alpha = [\overline{3}] = \frac{3 + \sqrt{13}}{2}$. The circles corresponding to the circle replacement algorithm are shaded gray.}
	\label{fig:selfsim3}
\end{figure}

\begin{figure}
\begin{overpic}[scale = .75]{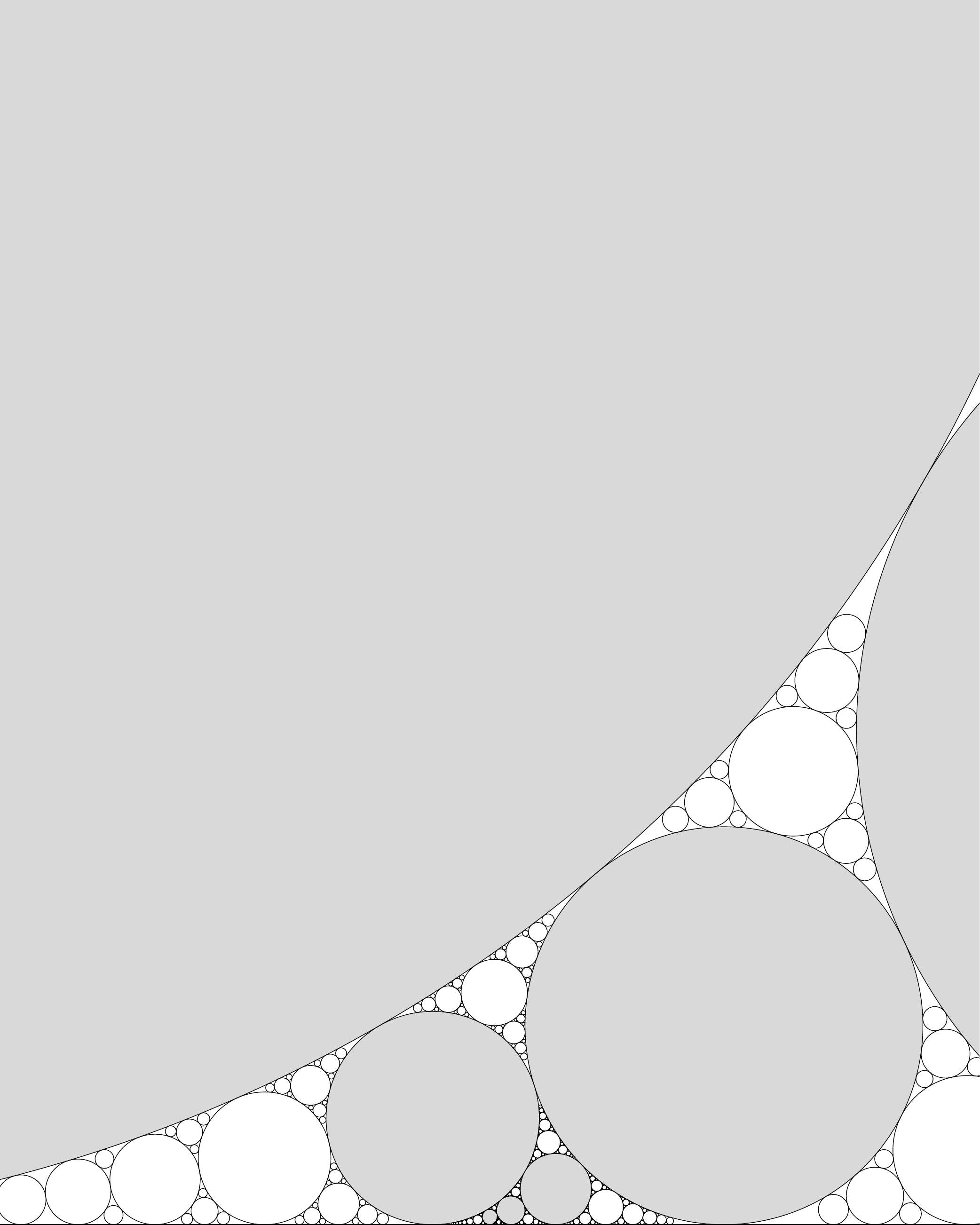}
	\put(223,-10){\large $\uparrow$}
	\put(237,-10){\large $\nwarrow$}
	\put(218,-24){\Large $X_0$}
	\put(248,-24){\Large $Y_0$}
\end{overpic}
	
\vspace{10mm}	
	
	\caption{Eight generations of the packing $\PP_{\alpha}$, where $\alpha = [\overline{1,2}] = \frac{1 + \sqrt{3}}{2}$. The circles corresponding to the circle replacement algorithm are shaded gray.}
	\label{fig:selfsim12}
\end{figure}

We have seen that the self-similarity of the packing $\PP_\alpha$ is reflected in the continued fraction expansion of $\alpha$. In fact, it turns out that every self-similarity of $\PP_{\alpha}$ comes about from the periodicity of the continued fraction expansion. In particular, the self-similarity groups of the packings in Figures \ref{fig:selfsim1}-\ref{fig:selfsim12} are generated by those arising from the circle replacement algorithm. We can make this more precise as follows.

Suppose for simplicity that $\alpha$ has a \emph{purely periodic} continued fraction expansion:
\[ \alpha = [a_0,\dots,a_{n-1},a_0,\dots,a_{n-1},a_0,\dots] = [\overline{a_0,\dots,a_{n-1}}]. \]
(By Theorem \ref{thm:eventuallyequal}, we lose no generality in doing so.) After performing $N = a_0 + \dots + a_{n-1} + n$ steps of the continued fraction algorithm, the initial segment consisting of the periodic part of the continued fraction expansion is removed; therefore $\alpha_N = \alpha$. It follows by Lemma \ref{lem:cfe} that applying $N$ steps of the circle replacement algorithm to the pair $(X_0,Y_0) = (C_{(1,0)},C_{(0,1)})$ yields a pair $(X_N,Y_N)$ whose curvatures are in the same ratio (i.e., $\alpha^2$) as the original circles $(X_0,Y_0)$. There is therefore a similarity of $\PP_{\alpha}$ that maps $X_0$ to $X_N$ and $Y_0$ to $Y_N$ by Lemma \ref{similar}. (Note that if the continued fraction expansion of $\alpha$ is not purely periodic, the above argument determines instead a similarity mapping $(X_M,Y_M)$ to $(X_N,Y_N)$ for some $M,N$.) In fact, this argument may be repeated to show that there is a similarity of $\PP_{\alpha}$ that maps $X_0$ to $X_{kN}$ and $Y_0$ to $Y_{kN}$ for each $k \in \zz_{\ge 0}$. The next lemma tells us that every self-similarity of $\PP_{\alpha}$ arises from the continued fraction expansion of $\alpha$ in this way.

\begin{lem}\label{lem:selfsimcfe}
Suppose $\alpha$ is the positive real number with periodic continued fraction expansion $[\overline{a_0,\ldots,\alpha_{n-1}}]$ and that this is the minimal periodic part. Let $\Phi$ denote the generator of $\Symm(\PP_{\alpha}) \cong \mathbb{Z}$ for which $\Phi$ has scale factor greater than 1, and let $(X_0,Y_0) = (C_{(1,0)},C_{(0,1)})$ be the generators for $\PP_{\alpha}$. Then
\[ \Phi^{k}(X_0) = X_{kN} \ , \ \Phi^{k}(Y_0) = Y_{kN}\]
for each $k \geq 0$, where $N = a_0 + \cdots + a_{n-1} + n$ as above, and $\Phi^{k}$ denotes the $k$-fold composition of $\Phi$ with itself (and $\Phi^{0}$ is the identity map on $\rr^2$).
\end{lem}

\begin{proof}
We have already argued that for each $k \ge 0$ there is a self-similarity $\Phi_k$ of $\PP_{\alpha}$ such that
$$\Phi_k(X_0) = X_{kN} \mbox{ \ \ \ and \ \ \ } \Phi_k(Y_0) = Y_{kN}.$$
Furthermore, if $N$ does not divide $m$, then $\alpha_m \ne \alpha$, so that the \emph{only} self-similarities that correspond to the circle replacement algorithm are the $\Phi_k$.

Because the radius of $Y_0$ is 1, the scale factor $\mu_k$ of $\Phi_k$ is equal to the radius of $\Phi_k(Y_0) = Y_{kN}$, which increases as $k$ increases by the definition of the circle replacement algorithm. (Note that the $Y_{kN}$ must necessarily be distinct.) Similarly, since the scale factor for $\Phi^{k}$ is $\mu^k$, with $\mu > 1$, the scale factor for $\Phi^{k}$ also increases as $k$ increases. It will therefore suffice to show that each self-similarity $\phi$ with scale factor greater than 1 satisfies
$$\phi(X_0) = X_m \mbox{ \ \ \ and \ \ \ } \phi(Y_0) = Y_m$$
for some $m \ge 1$.

By Lemma \ref{lem:replacementcircles}, we may reduce this problem to showing that each of $X_0$ and $Y_0$ either lies in the bounded interstice for $\phi(X_0)$ and $\phi(Y_0)$ or is equal to one of $\phi(X_0)$ and $\phi(Y_0)$. Because self-similarities preserve the basic structure of the packing, $X_0$ (resp. $Y_0$) lies in the bounded interstice for $\phi(X_0)$ and $\phi(Y_0)$ if and only if $\phi^{-1}(X_0)$ (resp. $\phi^{-1}(Y_0)$) lies in the bounded interstice for $X_0$ and $Y_0$. It is therefore enough to show that if $\psi$ is a self-similarity of $\PP_{\alpha}$ with scale factor \emph{less than} 1, then $\psi(X_0)$ lies in the bounded interstice for $X_0$ and $Y_0$. In that case $\psi(Y_0)$ will have to either be in the bounded interstice, or be equal to $X_0$ or $Y_0$.

By the proof of Theorem \ref{thm}, for any nontrivial self-similarity $\psi$ of $\PP_{\alpha}$, we have
\begin{align*}
\psi(X_0) = \psi(C_{(1,0)}) &= C_{\left(\frac{x-yq}{2},-yr\right)}\\
\psi(Y_0) = \psi(C_{(0,1)}) &= C_{\left(yp,\frac{x+yq}{2}\right)},
\end{align*}
where $x$ and $y$ are integers satisfying $|x^2 - Dy^2| = 4$ with $y \neq 0$. Here $f(x) = px^2 + qx + r$, $p > 0$, is the primitive integer polynomial satisfied by $\alpha$, and $D = q^2 - 4pr$ is the discriminant of $f$.

To show that $\psi(C_{(1,0)})$ either lies in the bounded interstice for $C_{(1,0)}$ and $C_{(0,1)}$, or is equal to $C_{(0,1)}$, it will suffice by Lemmas \ref{boundedpositive} and \ref{inversion} to show that
\begin{equation}\label{eqn:endproof}
\frac{x-yq}{2} > 0 \mbox{ \ \ \ and \ \ \ } -yr > 0.
\end{equation}

To see this, we first recall a result due to Galois \cite{davenport:2003,galois:1829} concerning purely periodic continued fractions. Because $\alpha$ has a purely periodic continued fraction expansion, it is a reduced quadratic number; i.e., $\alpha > 1$ and $-1 < \alpha' < 0$, where $\alpha'$ is the quadratic conjugate of $\alpha$. Since $q = -p(\alpha + \alpha')$ and $r = p\alpha\alpha'$, and since $p > 0$, it follows that $q,r < 0$. It is now sufficient to show that $x > 0$ and $y >0$.

First note that we cannot have both $x,y < 0$ since then the label on $\psi(C_{(1,0)})$ would consist of two negative numbers, which is impossible by Proposition \ref{uniquecircle}. We now show also that $x$ and $y$ cannot have different signs.

Let $\mu'$ be the scale factor of $\psi$ which, by assumption, is less than $1$. Since $\curv(C_{(0,1)}) = 1$ and $\curv(\psi(C_{(0,1)})) = (yp\alpha + \frac{x+yq}{2})^2$, it follows that
\[\mu' = \frac{1}{\left(yp\alpha + \frac{x+yq}{2}\right)^2}.\]
Therefore
\[ \left|yp\alpha + \frac{x+yq}{2}\right| > 1. \]
Since $\alpha$ has purely periodic continued fraction expansion, it is greater than its conjugate, so we have $\alpha = \frac{-q + \sqrt{D}}{2p}$. The above inequality then implies
\[  |x + y\sqrt{D}| > 2. \]
Now $x^2 - Dy^2 = \pm 4$, so $x = \pm \sqrt{Dy^2 \pm 4}$. If $x$ and $y$ have different signs, then this inequality becomes
\[ |\sqrt{Dy^2 \pm 4} - \sqrt{Dy^2}| > 2. \]
But $|\sqrt{t + 4} - \sqrt{t}| \leq 2$ for all $t \geq 0$, so in fact $x$ and $y$ must have the same sign (and $y \neq 0$ since we assumed that $\psi$ was a nontrivial self-similarity). This completes the proof.
\end{proof}

Finally, we can also use continued fractions to see which packings have orientation-reversing self-similarities.

\begin{thm}
For a positive real number $\alpha$, the circle packing $\PP_{\alpha}$ has an orientation-reversing self-similarity if and only if the continued fraction expansion of $\alpha$ has odd period.
\end{thm}
\begin{proof}
We know from Lemma \ref{lem:selfsimcfe} that all self-similarities of $\PP_{\alpha}$ correspond to periods in the continued fraction expansion of $\alpha$. The self-similarity $\Phi$ is orientation-reversing if the corresponding pairs of circles in the circle replacement algorithm, say $(X_n,Y_n)$ and $(X_m,Y_m)$, for which $\Phi(X_n) = X_m$ and $\Phi(Y_n) = Y_m$, satisfy $X_n \prec Y_n$, but $Y_m \prec X_m$. Since the orientation of the circles $X_k,Y_k$ changes once for each case of step (B) in the algorithm, that is, for each term in the continued fraction expansion, we see that $\Phi$ is orientation-reversing if and only if it corresponds to an odd period.
\end{proof}

In the examples, we see that the circle packings $\PP_{\alpha}$ for $\alpha = [\bar{1}], [\bar{2}], [\bar{3}]$ do have orientation-reversing self-similarities, whereas that for $\alpha = [\overline{1,2}]$ does not.

\bibliographystyle{amsplain}
\bibliography{jdoyle}

\end{document}